\documentclass[12pt, a4paper, english]{amsart}
\usepackage{amsmath,amsthm,amssymb}
\usepackage{bbm}
\usepackage{libertine}
\usepackage{mathpazo}
\usepackage[shortlabels]{enumitem}
\usepackage{graphicx}
\usepackage{hyperref}
\usepackage{listings}
\usepackage{mathrsfs}
\usepackage{mathtools}
\usepackage{multicol}
\usepackage{stmaryrd}

\usepackage{tikz}
\usepackage{tikz-cd}
\usepackage{xcolor}
\definecolor{darkcandyapplered}{rgb}{0.64, 0.0, 0.0}
\definecolor{midnightblue}{rgb}{0.1, 0.1, 0.44}
\definecolor{mynewgreen}{HTML}{43916e}
\definecolor{lightblue}{RGB}{70, 130, 180}

\usepackage{hyperref}
\usepackage[noabbrev,capitalize]{cleveref}
\hypersetup{
	pdftoolbar=true,        
	pdfmenubar=true,        
	pdffitwindow=false,     
	pdfstartview={FitH},  % fits the width of the page to the window
	pdftitle={},
	pdfauthor={},
	pdfsubject={},
	pdfkeywords={},
	pdfnewwindow=true,  % links in new window
	colorlinks=true,  % false: boxed links; true: colored links
	linkcolor=darkcandyapplered,   % color of internal links
	citecolor=midnightblue,  % color of links to bibliography
	urlcolor=cyan,  % color of external links
	linktocpage=true  % changes the links from the section body to the page number
}

\newtheorem{thmx}{Theorem}

\newtheorem{conjx}[thmx]{Conjecture}

\theoremstyle{plain}
\newtheorem{theorem}{Theorem}[section]
\newtheorem{proposition}[theorem]{Proposition}
\newtheorem{lemma}[theorem]{Lemma}
\newtheorem{corollary}[theorem]{Corollary}
\newtheorem{conjecture}[theorem]{Conjecture}

\theoremstyle{definition}
\newtheorem{definition}[theorem]{Definition}
\newtheorem{question}[theorem]{Question}
\newtheorem{example}[theorem]{Example}
\theoremstyle{remark}
\newtheorem{remark}[theorem]{Remark}

\newcommand{\ZZ}{\mathbbmss{Z}} 
 
\newcommand{\RR}{\mathbbmss{R}}

\renewcommand{\to}{\rightarrow}

\renewcommand{\emptyset}{\varnothing}
\newcommand{\conv}{\mathrm{conv}} % convex hull
\newcommand{\link}{\mathrm{link}} % link

\newcommand{\lk}{\mathrm{lk}} % link

\newcommand{\cy}{\mathrm{cy}}

\newcommand{\mysetminusD}{\hbox{\tikz{\draw[line width=0.6pt,line cap=round] (3pt,0) -- (0,6pt);}}}
\newcommand{\mysetminusT}{\mysetminusD}
\newcommand{\mysetminusS}{\hbox{\tikz{\draw[line width=0.45pt,line cap=round] (2pt,0) -- (0,4pt);}}}
\newcommand{\mysetminusSS}{\hbox{\tikz{\draw[line width=0.4pt,line cap=round] (1.5pt,0) -- (0,3pt);}}}
\newcommand{\sm}{\mathbin{\mathchoice{\mysetminusD}{\mysetminusT}{\mysetminusS}{\mysetminusSS}}}

\usepackage{mathtools} % for the absolute value symbol
\DeclarePairedDelimiter\abs{\lvert}{\rvert}%
\makeatletter
\let\oldabs\abs
\def\abs{\@ifstar{\oldabs}{\oldabs*}}

\newcommand\restricted[2]{{% we make the whole thing an ordinary symbol
  \left.\kern-\nulldelimiterspace % automatically resize the bar with \right
  #1 % the function
  \littletaller % pretend it's a little taller at normal size
  $\right|_{#2}$ % this is the delimiter
  }}
\newcommand{\littletaller}{\mathchoice{\vphantom{\big|}}{}{}{}}

\title{The number of edges of a symmetric edge polytope}
\author{Giulia Codenotti}
\address{(G. Codenotti)
	Department of Mathematics, Freie Universit\"at Berlin, Berlin, Germany
}
\email{giulia.codenotti@fu-berlin.de}

\author{Roberto Riccardi}
\address{(R. Riccardi)
	Scuola Normale Superiore, Pisa, Italy
}
\email{roberto.riccardi@sns.it}

\author{Lorenzo Venturello}
\address{(L. Venturello)
	Dipartimento di ingegneria dell'informazione e scienze matematiche, Universit\`a di Siena, Siena, Italy
}
\email{lorenzo.venturello@unisi.it}

\makeindex

\begin{document}
\begingroup
\maketitle
\endgroup

\begin{abstract}
    The symmetric edge polytope of a simple graph is a lattice polytope defined as the convex hull of a subset of the type $A$ roots corresponding to the edges of the graph. In this article we prove a sharp lower bound for the number of edges of the symmetric edge polytope of a graph as a function of elementary graph invariants. Moreover, we characterize graphs attaining this bound. We highlight a connection with the $h^*$-polynomial of such polytopes and, motivated by a conjecture of Ohsugi and Tsuchiya, we investigate the behaviour of such polynomial under edge-deletion in the graph.
\end{abstract}

\section{Introduction}

The main object of study of this article is the \emph{symmetric edge polytope} (in short SEP) of a graph. For a fixed finite simple graph $G = ([n],E)$, the symmetric edge polytope of $G$ is defined as $P_G := \conv\{\pm(e_i - e_j) : \{i,j\} \in E\}\subseteq \mathbb{R}^n$. These polytopes are a rather popular tool arising, sometimes in disguise, in different areas of mathematics such as physics (see for instance \cite{CDM,DDM}), the study of discrete metric spaces (\cite{Ver,DeHo}) and optimal transport (\cite{WassDist}). Our focus lies on combinatorial and geometric properties of these objects, a point of view which has also recently attracted a lot of attention (see \cite{HJM19,DDM,CDE,KT1,MMO}).\\
Our first result is a lower bound on the number of edges of the symmetric polytope of a graph $G$, denoted by $f_1(P_G)$ as a function of elementary graph invariants of $G$.
\begin{thmx}(\Cref{thm:edges_SEP_ineq})\label{thm: A}
	Let $G=(V,E)$ be a connected graph. Then 
    \[
    f_1(P_G) \geq  |E|(2|V|-5) + |E_3(G)|,
    \]
    with $E_3(G):= \lbrace e \in E \, | \, e \text{ belongs to a 3-cycle of $G$} \rbrace$.
\end{thmx}
The proof of \Cref{thm: A} is graph theoretical, and based of a delicate analysis of the local contribution of each edge of the graph $G$ developed in \Cref{sec:z_2(G)}.\\
Next we show that the lower bound in \Cref{thm: A} is sharp and we obtain a rather clean characterization of graphs for which the bound in \Cref{thm: A} is attained.

\begin{thmx}(\Cref{cor: Z = 0})\label{thm: B}
	Let $G=(V,E)$ be a connected graph. Then
    \[
    f_1(P_G) =  |E|(2|V|-5) + |E_3(G)|
    \]
    if and only if either $G\cong K_n$, $G\cong K_{1,1,n-2}$ or $G\cong K_{2,n-2}$.
\end{thmx}
We find a relation between the number $f_1(P_G)$ and the \emph{Ehrhart theory} of $P_G$. An interesting invariant for any lattice polytope $P$ is the function which controls how the number of lattice points grows with the dilations of the polytope. By a seminal result of Ehrhart this function agrees with a polynomial, and is called the Ehrhart polynomial of $P$. It is often more convenient to study the rational function obtained as the associated generating function to the Ehrhart polynomial. Its numerator is the \emph{$h^*$-polynomial} of $P$, denoted by $h^*_P(t)$. The $h^*$-polynomial of $P$ is a polynomial with nonnegative integer s and degree at most the dimension of $P$. In the case of symmetric edge polytopes the $h^*$-polynomial has degree exactly the dimension of the polytope, and a very special property: it is \emph{palindromic}, its coefficients satisfy $h^*_i=h^*_{d-i}$, where $h^*_{P_G}(t)=\sum_{i=0}^{d}h^*_it^i$ and $d=\dim P_G$.
Our main goal is to investigate a conjecture of Ohsugi and Tsuchiya \cite{OT21} predicting certain linear inequalities among the coefficients of the $h^*$-polynomial of symmetric edge polytopes. In order to state the conjecture we need to rewrite $h^*_{P_G}(t)$ in a particular basis for the space of palindromic polynomials of fixed center of symmetry (see \Cref{sec: preliminaries} for more details). The coefficients of $h^*_{P_G}(t)$ with respect to such basis can be collected to define a polynomial $\gamma_{P_G}(t)=\sum_{j=0}^{\lfloor\frac{d}{2}\rfloor} \gamma_jt^j$ called the \emph{$\gamma$-polynomial} associated to $h^*_{P_G}(t)$.

\begin{conjecture}{\cite[Conjecture 5.11]{OT21}}\label{conj: OT} The $\gamma$-polynomial associated to $h^*_{P_G}(t)$ has nonnegative coefficients for every graph $G$.   
\end{conjecture}
The $\gamma$-polynomial has been introduced by Gal \cite{Gal} to formulate a different but related positivity conjecture. Moreover, nonnegativity of $\gamma$-polynomials has received a lot of attention in combinatorics, as this property sits between the weaker unimodality of the coefficients of the corresponding palindromic polynomial and the stronger real-rootedness of the coefficients (see \cite{Ath, Bra}).\\
Interestingly, the polynomial $h^*_{P_G}(t)$, and hence $\gamma_{P_G}(t)$ only depend on the graphic matroid of $G$ (see for instance \cite[Theorem 4.6]{DJK}). In \cite{DJK} the authors generalize the definition of symmetric edge polytopes to the broader class of regular matroids, and show that these new objects still have a palindromic $h^*$-polynomial. However, it was recently shown that there are regular matroids for which the corresponding $\gamma$-polynomial has a negative coefficient.\\
\Cref{conj: OT} has been verified for either very special families of graphs (\cite{OT-2020,OT21}) or for specific coefficients of $\gamma_{P_G}(t)$ (\cite{AJKKV23}). More precisely, $\gamma_0=1$ and $\gamma_1$ is equal to twice the cyclomatic number of the graph. In \cite{AJKKV23} the authors show that also $\gamma_2$ is nonnegative for every graph, even though there is no elementary interpretation of this number. Their proof strategy relies on a comparison between the quadratic coefficients of $\gamma_{P_G}(t)$ and of $\gamma_{P_{G\setminus e}}(t)$. In particular, they show that for every graph $G$ \emph{there exists} an edge $e$ such that the quadratic coefficient of $\gamma_{P_G}(t)-\gamma_{P_{G\setminus e}}(t)$ is nonnegative. This observation, together with a simple inductive argument, yields the nonnegativity of $\gamma_2$ for every graph (see \Cref{prop:goal_implies_gamma})). However, there are graphs and choices of an edge $e$ for which the quadratic coefficient of $\gamma_{P_G}(t)-\gamma_{P_{G\setminus e}}(t)$ is negative (see \Cref{ex:negative-gamma}). In order to absorb the potentially negative contribution we propose the following definition.

\begin{definition}
    Let $G=([n],E)$ be a $2$-connected graph. We define
    \[
        Z_G(t):= \sum_{e\in E}(\gamma_{P_G}(t)-\gamma_{P_{G\setminus e}}(t)).
    \]
\end{definition}
In order to prove \Cref{conj: OT} it would be sufficient (in fact stronger) to show the following conjecture.
\begin{conjx}(Conjecture \ref{conj: zG nonnegative})
	Let $G=([n],E)$ be a $2$-connected graph. Then the polynomial $Z_G(t)$ has nonnegative coefficients.
\end{conjx}

 While it is rather immediate to show that the constant and linear coefficient are nonnegative, we make use of \Cref{thm: A} to prove the following.

\begin{thmx}(Equation \eqref{eq: z2 is z2})\label{thm: C}
	Let $G=([n],E)$ be a $2$-connected graph. Then 
    \[
        z_2 = f_1(P_G) -  |E|(2|V|-5) - |E_3(G)|,
    \]    
    where $z_2$ is the quadratic coefficient of $Z_G(t)$. In particular, it follows from \Cref{thm: A} that $z_2\geq 0$.
\end{thmx}

Perhaps surprisingly, we derive the equality in \Cref{thm: C} from a geometric source. First we consider a particular triangulation of $P_{G}$ into lattice simplices, and then we employ Betke-McMullen formula to compare $h^*_{P_G}(t)$ and $h^*_{P_{G\setminus e}}(t)$. Besides the result in \Cref{thm: C} we propose in \Cref{sec:conjectured_formula} a conjectured formula for the polynomial $h^*_{P_G}(t)- h^*_{P_{G\setminus e}}(t)$ which can provide further insights on \Cref{conj: OT} (see \Cref{cong: formula for difference}).

The structure of the article is the following: in \Cref{sec: preliminaries} we provide the reader with all necessary definitions and background, as well as important results about symmetric edge polytopes. \Cref{sec:z_2(G)} and \Cref{sec:z_2(G)=0} are devoted to the proofs of \Cref{thm: A} and \Cref{thm: B} respectively. In \Cref{sec:Ehrhart} we explain the connection with Ehrhart theory and we prove \Cref{thm: C}. Finally in Section \Cref{sec:conjectured_formula} we formulate new conjectures on the $h^*$-polynomial of symmetric edge polytopes.

\section{Preliminaries}\label{sec: preliminaries}

\subsection{Lattice polytopes}

A \emph{lattice polytope} is the convex hull of finitely many points which are in a lattice contained in $\RR^d$, which is typically $\ZZ^d$. 
An important property of lattice polytopes is reflexivity, which is first defined in the full-dimensional case, then extended in the general framework. When $P \subseteq \RR^d$ is full-dimensional we say that $P$ is \emph{reflexive} if 
\[
    P ^\vee := \left\{ u \in \RR^d \, | \, \langle u,v\rangle \leq 1 \text{ for any } v \in P  \right\}
\]
is also a lattice polytope, where $\langle \cdot , \cdot \rangle$ is the usual scalar product in $\RR^d$. In general, a lattice polytope $P\subseteq \RR^d$ is \emph{reflexive} if there exists a linear isomorphism $\varphi: \text{aff($P$)} \longrightarrow \RR^m$ which is also a lattice isomorphism such that $\varphi(P) \subseteq \RR^m$ is a full-dimensional lattice polytope and it is reflexive.

For any lattice $d$-polytope $P$ we consider $E_P(n):= | nP \cap \ZZ^d |$. Ehrhart in \cite{Ehrhart} proved that $E_P(n)$ is a polynomial in $n$ of degree $d$, the \emph{Ehrhart polynomial}. Equivalently, the Ehrhart series $Ehr_P(t):= 1 + \sum_{n\geq 1 } { E_P(n)t^n }$ is a rational function of the form $\frac{h^*_P(t)}{(1-t)^{d+1}}$, where $h^*_P(t)$ is a polynomial of degree at most $d$, which is called the \emph{$h^*$-polynomial} of $P$.

A fundamental theorem by Stanley \cite{Stanley} states that the $h^*$-polynomial has non negative integer coefficients, and Hibi in \cite{Hib} proved that a $d$-dimensional lattice polytope $P$ is reflexive if and only if its $h^*$-polynomial is palindromic, that is, $h^*_P(t)=t^dh^*_P(\frac{1}{t})$.

\subsection{Triangulations and relations with Ehrhart theory}
A \emph{triangulation} $\mathcal{T}$ of a d-polytope $P$ is a simplicial complex whose geometric realization is $P$. The \emph{f-polynomial} $f_{\mathcal{T}}(t)=f_{-1} + f_0x+ \ldots + f_dx^{d+1}$ encodes the number of faces of $\mathcal{T}$ in all dimensions: $f_i := \big| \lbrace \Delta \in \mathcal{T} \, | \, \dim \Delta= i \rbrace \big|$. The \emph{h-polynomial} $h_{\mathcal{T}}(t)= h_0+ h_1x+ \ldots + h_{d+1}x^{d+1}$ is given via the following relation:
\[
	f_{\mathcal{T}}(t)= \sum_{i=0}^{d+1} h_i t^i (1+t)^{d+1-i}.
\]

A d-dimensional lattice simplex is called \emph{unimodular} if its vertices affinely span the integer lattice $\ZZ^d$ and a triangulation of a lattice polytope into unimodular simplices is called a \emph{unimodular triangulation}. Stanley in \cite{Stanley} gave an important result which relates lattice polytopes and their unimodular triangulations.

\begin{theorem}{\cite[Corollary 2.5]{Stanley}}
	If $P$ is an integral d-polytope that admits a unimodular triangulation $\mathcal{T}$, then $h^*_P(t)= h_{\mathcal{T}} (t)$.
\end{theorem} 

We will further make use of the following theorem, known as the Betke-McMullen formula. Indeed, Betke and McMullen in \cite{BetkeMcMullen} extend the notion of $h^*$-polynomial to apply to the support of a simplicial complex, and give a nice formula which extends the previous result of Stanley. 

\begin{theorem}{\cite[Theorem 1]{BetkeMcMullen}}\label{BetkeMcMullen}
	Let $\Delta$ be a geometric simplicial complex whose vertices are in $\ZZ^d$, and let $|\Delta|$ denote its support. Then
	
	\[
	h^*_{|\Delta|}(z) = \sum_{\sigma \in \Delta} h_{\link(\sigma)}(z)\ell^*_\sigma(z),
	\]
	
\end{theorem}

where $\ell^*_\sigma(z)$ is the local $h^*$-polynomial (in the case of simplices is also called the box polynomial), a related notion to the $h^*$-polynomial of a lattice polytope. Defining it for general polytopes has its subtleties (see \cite{BKN24}), but for simplices there is a nice combinatorial interpretation (\cite{Stan92}), which will serve as definition here.

\begin{definition}
	Let $S\subseteq \RR^d$ be a $k$-dimensional lattice simplex with vertices $v_0, \dots, v_k$.
	The fundamental parallelepiped $\Pi(S)$ of $S$ is the following $(k+1)-$dimensional polytope,
	
	\[
	\Pi(S)= \sum_{i =0}^{k+1} \conv(0, \bar{v_i}),
	\]
	where the sum denotes the Minkowski sum and $\bar{v_i}= (v_i, 1) \in \RR^{d+1}$.

		The open fundamental parallelepiped  $\overset{\circ}{\Pi(S)}$ is the interior of $\Pi(S)$. Then the local $h^*$ polynomial of $S$ is defined as $\ell_S^*(t)= c_0(S) + c_1(S)t + \dots + c_{k+1}(s)t^{k+1}$, where 
	\[c_i(S) =  |\overset{\circ}{\Pi}(S) \cap \{\bar{x} \in \ZZ^{d+1} \, | \, \bar{x}_{d+1}=i \} |,\]
	that is, $c_i$ counts integer points with $(d+1)$-th coordinate equal to $i$ inside the open fundamental parallelepiped $\overset{\circ}{\Pi}(S)$.
\end{definition}

\subsection{Basics of symmetric edge polytopes}
\begin{definition}
Given a finite, simple graph $G = ([n],E)$, the \emph{symmetric edge polytope} (SEP) is defined to be
\[
P_G := \conv\{\pm(e_i - e_j) : ij \in E\}.
\]		
\end{definition}
We define $e_{ij}:=e_i-e_j$. Symmetric edge polytopes live on the hyperplane 
\[
H= \lbrace x_1 + x_2 + \cdots + x_n = 0 \rbrace,
\]
and their dimension is $n-c$, where $c$ is the number of connected components of $G$. We can identify each vertex of $P_G$ with an edge of $G$ with an assigned orientation. In this way every face of $P_G$, or a face of a triangulation of $P_G$ which does not introduce new vertices, can be identified with an oriented subgraph of $G$. We denote by $G(F)$ the oriented subgraph corresponding to a face $F$.

In \cite{FirstSEP}, Proposition 3.2, it was proved that any symmetric edge polytope $P_G$ of a finite simple connected graph $G$ is reflexive and \cite{HJM19} showed also that they also admit regular unimodular triangulations, so its $h^*$-polynomial is palindromic and unimodal.

Moreover, since $P_G$ is reflexive, every affine hyperplane which defines a facet of $P_G$ is of the form 
\[
H_f=\Big\{ x \in H \, | \, \sum_{v \in [n]} f(v)x_n=1 \Big\},
\]

where $f: [n] \longrightarrow \ZZ $ is called \emph{facet defining}.

In \cite{HJM19}, Higashitani, Jochemko and Michalek give the characterization of the facets of $P_G$ in terms of $G$ when $G$ is connected.

\begin{theorem}{\cite[Theorem 3.1]{HJM19}}\label{thm:facetchar}
Let $G=(V,E)$ be a finite simple connected graph. Then $f\colon V\rightarrow \mathbb{Z}$ is facet defining if and only if
\begin{itemize}
	\item[(i)] for any edge $e=uv$ we have $\abs{f(u)-f(v)}\leq 1$, and
	\item[(ii)] the subset of edges $E_f=\{e=uv\in E \colon \abs{f(u)-f(v)}=1\}$ forms a spanning subgraph of $G$.
\end{itemize}
\end{theorem}

A consequence of the previous theorem is the following. 

\begin{corollary}{\cite[Corollary 3.2]{HJM19}}\label{cor:unimodsimplexchar}
	The unimodular simplices contained in a facet $F$ of $P_G$ represented by a function $f$ correspond exactly to oriented spanning trees of $G(F)$, that is, oriented spanning trees consisting of oriented edges $\vec{vw}$ such that $f(v)-f(w)=1$.
\end{corollary}

Further, in \cite{HJM19}, it is shown that symmetric edge polytopes have regular unimodular
triangulations. We will outline their construction here and call these triangulations
\emph{HJM triangulations}. The construction is given in terms of its Stanley-Reisner ideal: that
is, a squarefree ideal $J$ in the ring $K[x_{\alpha} \, : \, \alpha \in P_G \cap \mathbb{Z}^d]$, and the corresponding triangulation $\Gamma$ of $P_G$ consists of all the subsets $S \subseteq P_G \cap \mathbb{Z}^d$ such that conv($S$) is a simplex and $\prod_{\alpha \in S} x_{\alpha} \not \in J$.

In order to construct such $J$, remember at first that the integer points of $P_G$ are precisely $\lbrace 0 \rbrace \cup \lbrace  \pm e_{ij} \, : \, ij \in E(G) \rbrace$, so $K[x_{\alpha} \, : \, \alpha \in P_G \cap \mathbb{Z}^d]$ is naturally identified as 
$K[\lbrace x_e, y_e \rbrace_{e \in E(G)} \cup \lbrace z \rbrace]$ where $z$ is the variable associated to the origin.

In order to simplify notation, in the following, for any oriented edge $e$, we denote by $p_e$ the corresponding variable, i.e.~$p_e=x_e$ or $p_e=y_e$ depending on the orientation. We also set $q_e$ to be equal to the variable with the opposite orientation, i.e.~$\{p_e,q_e\}=\{x_e,y_e\}$.

Then, fix a total order $<$ on $E(G)$ that induces a total order on the variables in $K[\lbrace x_e, y_e \rbrace_{e \in E(G)} \cup \lbrace z \rbrace]$, hence consider the degrevlex order respect to $<$ on the monomials.
$J$ will be the initial ideal respect to $<$ of a toric ideal $I_{P_G}$ which Gr\"obner basis is given by the following Proposition.

\begin{proposition}{\cite[Proposition 3.8]{HJM19}} \label{thm:HJMtriang}
	Let $z< x_{e_1}<y_{e_1}<\dots<x_{e_k}<y_{e_k}$ be an order on the edges. Then the following collection of three types of binomials forms a Gr\"obner basis of the toric ideal of $P_G$ with respect to the degrevlex order:
	\begin{enumerate}
		\item For every $2k$-cycle $C$, with fixed orientation, and any $k$-element subset $I$ of edges of $C$ not containing the smallest edge  
		$$\prod_{e\in I}p_e-\prod_{e\in C\sm I} q_e.$$
		\item For every $(2k+1)$-cycle $C$, with fixed orientation, and any $(k+1)$-element subset $I$ of edges of $C$ 
		$$\prod_{e\in I}p_e-z\prod_{e\in C\sm I}q_e.$$
		\item For any edge $e$ $$x_ey_e-z^2 \, .$$
	\end{enumerate}
	The leading monomial is always chosen to have positive sign.
\end{proposition}

It is important to note that every maximal simplex of the HJM triangulation contains
the origin. In fact, given a monomial $m$ without the variable $z$, we have

\[
m \in \text{in}_< (I_{P_G}) \Longleftrightarrow z \cdot m \in \text{in}_< (I_{P_G}).
\]

\subsection{Real-rootedness, unimodality and $\gamma$-positivity}
Let $a_0, a_1, \dots, a_d$ be a sequence of non-negative real numbers.
The polynomial $f(t)=\sum_{i=0}^d a_it^i$ of degree $d$ is \emph{real-rooted} if all its roots are real. If $f$ is real-rooted, then the sequence of the coefficients $(a_i)_{i=0}^d$ is \emph{log-concave}, which means that $a_i^2 \geq a_{i-1}a_{i+1}$ for all $i = 1,\ldots, d-1$ \cite{Sta89}. A non-negative log-concave sequence is also \emph{unimodal}, that is, $a_0 \leq a_1 \leq \cdots \leq a_k \geq a_{k+1} \geq \cdots \geq a_d$ for some $k$.

The polynomial $f$ is said to be \emph{palindromic} if $f(t)=t^df(\frac{1}{t})$. A palindromic polynomial can be written as follows, for unique real numbers $\gamma_0, \gamma_1, \dots, \gamma_{\lfloor d/2 \rfloor}$: 

\[
    f(t)= \sum_{i=0}^{\lfloor d/2 \rfloor} \gamma_i t^i(1+t)^{d-2i}.
\]
The polynomial $\sum_i \gamma_i t^i$ is called the \emph{$\gamma$-polynomial} of $f$. We say that the polynomial $f(t)$ is $\gamma$\emph{-nonnegative} if $\gamma_0, \gamma_1, \dots, \gamma_{\lfloor d/2 \rfloor} \geq 0$.
We have that a $\gamma$-nonnegative polynomial is real-rooted if and only if its $\gamma$-polynomial is real-rooted.

It can be seen also that a $\gamma$-nonnegative polynomial is unimodal, so in the palindromic case, even if $\gamma$-positivity and real-rootedness are different notions, they share some nice combinatorial properties.

\section{Lower bound on the number of edges of symmetric edge polytopes}\label{sec:z_2(G)}
This section is devoted to proving a lower bound on  $f_1(P_G)$, the number of edges of the symmetric edge polytope of a connected graph $G = (V, E)$.
We let 
\[
E_3(G):= \lbrace l \in E \, | \, l \text{ is in a 3-cycle of $G$} \rbrace.
\]

We will prove the following lower bound on $f_1(P_G)$.

\begin{theorem} \label{thm:edges_SEP_ineq}
	Let $G$ be a connected graph. Then 
    \[
    f_1(P_G) \geq  |E|(2|V|-5) + |E_3(G)|.
    \]
\end{theorem}

As a first step, we characterize  $f_1(P_G)$  purely in terms of the graph. We use the notation $OE(G)$ to denote the set of all possible orientations of edges of $G$.

\begin{lemma}
    For a connected graph $G$, we have 
    \[
    f_1(P_G) =  |\underbrace{\bigl\{ \lbrace \vec{l_1}, \vec{l_2} \rbrace \, | \, \vec{l_1}, \vec{l_2} \in OE(G), \, l_1 \neq l_2 \text{ , $\vec{l_1}, \vec{l_2}$ not in a common oriented 3 or 4-cycle of $G$} \bigr\}}_{=:F_1(G)}|.
    \]

\end{lemma}
\begin{proof}
    Consider two vertices $e_{ij}$ and $e_{kl}$ of $P_G$, which correspond to two oriented edges $\vec{l_1} = (i,j)$ and $\vec{l_2}=(k,l)$ of $G$. These two vertices form an edge of $P_G$ if there is a hyperplane $H_c = \{x \in \mathbb{R}^n \mid c^T\cdot x = b\}$ which supports $P_G$ and contains $e_{ij}$ and $e_{kl}$ and no other vertex of $P_G$. Since $c^T \cdot e_{st} = c_t -c_s$ for any vertex $e_{st}$, we need to prove that there exists a vector $c$ and scalar $b$ which satisfy 

    \begin{align}\label{edges_sep}
    & c_j-c_i = c_l - c_k = b \\
    & c_s-c_t < b \text{ for all }e_{st} \text{ other vertices of }P_G
    \end{align}
    exactly when the conditions above are satisfied, that is, when the oriented edges $(i,j)$ and $(k,l)$ of the graph are not in oriented $3-$ or $4-$cycles.

    We distinguish two cases: when $i, j, k, l$ are $4$ distinct vertices of the graph $G$ and when two coincide to give us exactly $3$ distinct vertices of $G$.

    If $i, j, k, l$ are distinct, assume without loss of generality that they are $1, 2, 3, 4$; in this case, we have to show that the two oriented edges $(i,j)$ and $(k,l)$ are not in an oriented $4$-cycle.  In order to satisfy the conditions \ref{edges_sep}, $c$ must be of the form $(-a, b-a, a, b+a, c_5, \dots, c_n)$ (up to adding a constant to all entries), with $|c_i| < a$.
    Then to make sure that all vertices of $P_G$ except the two chosen ones are on the hyperplane, the only candidates are $(l, i)$ and $(j,k)$. Their scalar product with $c$ is $b-2a$ and $b+2a$, so both cannot correspond to vertices of $P_G$, as we wanted. 

    The case where we have three vertices $i, j, k$ is similar.    
\end{proof}

We can then reformulate \cref{thm:edges_SEP_ineq} by defining the function
\begin{equation}\label{def: z_2}
z_2(G):= f_1(P_G) - |E|(2|V|-5) - |E_3(G)|.
\end{equation}
The theorem is then equivalent to showing that $z_2(G) \geq 0$.

To prove this, we split $z_2(G)$ into a sum whose terms depend on single edges of $G$, and upper bound each of these terms. First of all, for any pair of edges $l_1$ and $l_2$ of $G$, we define
\[
	f(G,l_1,l_2):= \big| \bigl\{  \lbrace \vec{l_1}, \vec{l_2} \rbrace  \, | \, \vec{l_i} \text{ is an orientation of $l_i$ for $i=1,2$} \bigr\} \cap F_1(G) \big|
\]

\begin{remark}
	$f(G,l_1,l_2) \in \lbrace 0,2,4 \rbrace$ for any $l_1,l_2 \in E$. In particular:
	\begin{itemize}
		\item[\textbf{(i)} \hspace{1mm}] $f(G,l_1,l_2)=0$ if $l_1=l_2$, or if $l_1$ and $l_2$ are in a 4-clique of $G$ and they are not adjacent; 
		\item[\textbf{(ii)} ] $f(G,l_1,l_2)=2$ if $l_1$ and $l_2$ are in a 3-cycle of $G$, or the induced subgraph of $G$ respect to the vertices of $l_1$ and $l_2$ contains a 4-cycle but it's not a 4-clique; 
		\item[\textbf{(iii)}] $f(G,l_1,l_2)=4$ if $l_1$ and $l_2$ are not in any 3 or 4-cycle of $G$.
	\end{itemize}
\end{remark} 

Now, for any $l \in E$, we define the \emph{contribution of $l$ in $f_1(G)$ } as the set

\[
	F(G,l):= \bigl\{ \lbrace \vec{l_1}, \vec{l_2} \rbrace \in f_1(G) \, | \, \vec{l_1} \text{ is an orientation of } l \bigr\},
\]
and, setting $v_G:= |V(G)|$, we define the function $Z(G,l)$ as follows

\begin{equation} \label{zeta}
Z(G,l):= |F(G,l)| - 2(2v_G - 5) - 2 \cdot \mathbb{1}_{E_3(G)}(l)
\end{equation}
where $\mathbb{1}_{E_3}$ is the indicator function of $E_3$.

It's clear by the definitions that

\begin{equation} \label{sommacontributi}
	|F(G,l)|= \sum_{f \in E} f(G,l,f),
\end{equation}

and

\begin{equation} \label{z_2(G)}
	z_2(G)= \frac{1}{2} \sum_{l \in E} Z(G,l).
\end{equation}

What we will do is to show that $Z(G,l)$ has a good behavior, as stated in the following lemma.

\begin{proposition} \label{inequalityedge}
	Let $G$ be a connected graph. Then, for at most one $l \in E$ it holds that
	\begin{equation}
		Z(G,l) < 0,
	\end{equation}
	and if there is such an $l$, then there exists $l' \in E$ such that
	\begin{equation}
		Z(G,l) + Z(G,l') \geq 0.
	\end{equation}
\end{proposition}
In particular Proposition \ref{thm:edges_SEP_ineq} is proven if we apply Proposition \ref{inequalityedge} in \ref{z_2(G)}.

What we want to do is to "compute" $Z(G,l)$ in a recursive way, starting from the subgraph $G_0= \lbrace l \rbrace$ , and then construct a sequence of subgraphs of $G$
\[
\lbrace l \rbrace=G_0 \stackrel{m_{i_0}} \longrightarrow  G_1 \stackrel{m_{i_1}} \longrightarrow \cdots \stackrel{m_{i_{k-1}}} \longrightarrow G_k=G,
\] 
where $G_{i+1}$ is obtained by $G_i$ performing one of the following moves:

\begin{itemize}
	
	\item[\textbf{m1)} :] Adding a new vertex $w$ and a new edge $\lbrace v,w \rbrace$, where $v$ is a vertex of $G_i$. 
	\item[\textbf{m2)} :] Adding an edge of $G$ to $G_i$ with vertices in $V(G_i)$.
\end{itemize}

It's important to observe that \textbf{m1} preserves the Z function: in other words, if $G_i \stackrel{m1} \longrightarrow  G_{i+1}$, then $Z(G_{i+1},l)=Z(G_i,l)$. This is true because $v_{G_{i+1}}=v_{G_i}+1$ and the added edge $\lbrace v,w \rbrace$ is not in any cycle of $G_{i+1}$, giving us $f(G_{i+1},l,\lbrace v, w \rbrace)=4$. For \textbf{m2} the situation is more complicated: there are some cases in which, respectively, the Z function increases, decreases, or remains the same. Here some examples are reported.

	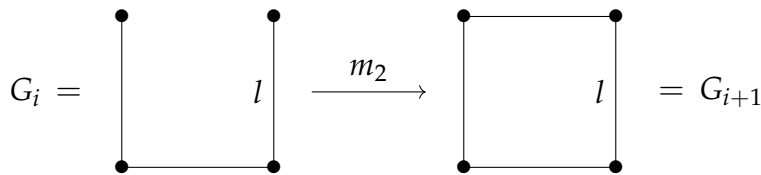
\begin{figure}[!ht]
		\begin{tikzpicture}[scale=0.5]
			
			\node at (0,0) 	   	 	{$\bullet$};
			\node at (0,4)			{$\bullet$};
			\node at (4,0) 			{$\bullet$};
			\node at (4,4) 			{$\bullet$};
			\node at (-2,2) 		{$G_i \, =$};
			
			\node at (9,0) 	   	 		{$\bullet$};
			\node at (9,4) 	   	 		{$\bullet$};
			\node at (13,0) 			{$\bullet$};
			\node at (13,4) 			{$\bullet$};
			\node at (15.5,2) 		{$= \, G_{i+1} $};
			
			\draw (5,2) edge[->] node[above] {$m_2$} (8,2);
						
			\path
			(4,0) edge node {\hspace{-4mm}$l$} (4,4)
			(4,0) edge node {} (0,0)
			(0,0) edge node {} (0,4);
			
			\path 
			(13,4) edge node {\hspace{-4mm}$l$} (13,0)
			(13,0) edge node {} (9,0)
			(9,0) edge node  {} (9,4)
			(9,4) edge node  {} (13,4);
		\end{tikzpicture}
		\caption{$Z$ decreases: $Z(G_i,l)=2$, meanwhile $Z(G_{i+1},l)=0$} \label{fig:decrescenza}
	\end{figure}
	
	\begin{figure}[!ht]
		\begin{tikzpicture}[scale=0.5]
			
			\node at (0,2) 	   	 	{$\bullet$};
			\node at (6,2)			{$\bullet$};
			\node at (3,0) 			{$\bullet$};
			\node at (3,4) 			{$\bullet$};
			\node at (-2,2) 		{$G_i \, =$};
			
			\node at (0+11,2) 	   	 	{$\bullet$};
			\node at (6+11,2)			{$\bullet$};
			\node at (3+11,0) 			{$\bullet$};
			\node at (3+11,4) 			{$\bullet$};
			\node at (19.5,2) 		{$= \, G_{i+1}$};
			
			\draw (7,2) edge[->] node[above] {$m_2$} (10,2);
			
			\path
			(3,0) edge node {\hspace{-4mm}$l$} (3,4)
			(3,0) edge node {} (0,2)
			(0,2) edge node {} (3,4)
			(6,2) edge node {} (3,4);
			
			\path 
			(3+11,0) edge node {\hspace{-4mm}$l$} (3+11,4)
			(3+11,0) edge node {} (0+11,2)
			(0+11,2) edge node {} (3+11,4)
			(6+11,2) edge node {} (3+11,4)
			(3+11,0) edge node {} (6+11,2);
		\end{tikzpicture}
		\caption{$Z(G_i,l)= 0 = Z(G_{i+1},l)$} \label{fig:uguaglianza}
	\end{figure}
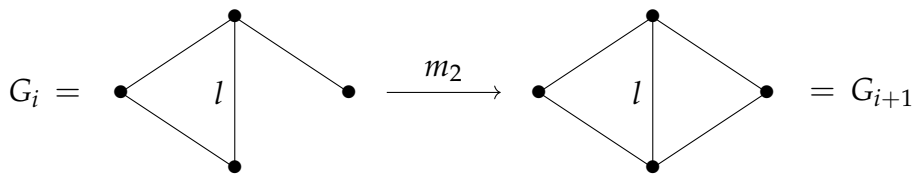
	
	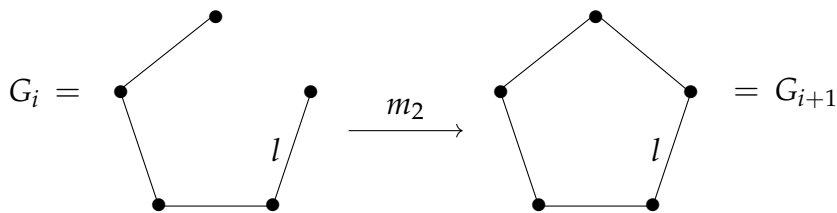
\begin{figure}[!ht]
		\begin{tikzpicture}[scale=0.5]
			
			\node at (0,0) 	   	 	{$\bullet$};
			\node at (3,0)			{$\bullet$};
			\node at (-1,3) 		{$\bullet$};
			\node at (4,3) 			{$\bullet$};
			\node at (1.5,5) 		{$\bullet$};
			\node at (-3,3)			{$G_i \, =$};
			
			\node at (0+10,0) 	   	 	{$\bullet$};
			\node at (3+10,0)			{$\bullet$};
			\node at (-1+10,3) 		{$\bullet$};
			\node at (4+10,3) 			{$\bullet$};
			\node at (1.5+10,5) 		{$\bullet$};
			\node at (6+10.5,3)			{$= \, G_{i+1}$};
			
			\draw (5,2) edge[->] node[above] {$m_2$} (8,2);
			
			\path
			(3,0) edge node {\hspace{-4mm}$l$} (4,3)
			(3,0) edge node {} (0,0)
			(0,0) edge node {} (-1,3)
			(-1,3.1) edge node {} (1.5,5.1);
			
			\path 
			(3+10,0) edge node {\hspace{-4mm}$l$} (4+10,3)
			(3+10,0) edge node {} (0+10,0)
			(0+10,0) edge node {} (-1+10,3)
			(-1+10,3.1) edge node {} (1.5+10,5.1)
			(4+10,3) edge node {} (1.5+10,5.1);
		\end{tikzpicture}
		\caption{$Z$ increases: $Z(G_i,l)=2$, meanwhile $Z(G_{i+1},l)=6$} \label{fig:decrescenza2}
	\end{figure}

 However, we can control the behavior of Z also for these kind of moves if we consider proper sequences of subgraphs, which are constructed as follows.
\begin{itemize}
	\item[(S1)] From $G_0$ perform \textbf{m1} for any vertex of $G$ adjacent to one of the vertices of $l$. At the end we obtain the subgraph $G^{(1)}$. 
	\item[(S2)] Add all the remaining edges incident to $l$ with moves \textbf{m2}, ending with a subgraph $G^{(2)}$.
	\item[(S3)] Add the rest of the edges with vertices in $V(G^{(1)})$ with moves \textbf{m2} , obtaining the subgraph $G^{(3)}$.
	\item[(S4)] Perform \textbf{m1} for any remaining vertex , obtaining $G^{(4)}$.
	\item[(S5)] Complete to $G$ adding the rest of the edges with moves \textbf{m2}.
\end{itemize}

What we do is to examine all the 5 steps described above. For what we said about \textbf{m1}, the steps (S1) and (S4) don't change the value of Z with respect to  $l$, so  $Z(G^{(1)},l)=Z(G_0,l)=2$ and $Z(G^{(4)},l)=Z(G^{(3)},l)$. Also for (S5) it's quite simple, because every edge $f$ added in this step doesn't create any new 3 or 4-cycle that contains $l$, so $f(G_{i+1},l,f)=4$ and the other contributions $f(G_{i+1},l,l')$ don't change, hence $Z(G_{i+1},l)= Z(G_i,l) + 4$. So at the end of the process we have
\[
	Z(G,l) = Z(G^{(4)},l) + 4(|E| - |E(G^{(4)})|).
\]

Next we study (S2). As it can be seen, all the 3-cycles of $G$ containing $l$ are constructed in this process, and moreover each move \textbf{m2} constructs a unique 3-cycle of this type. It's also important to notice that here we do not construct any other cycle of length greater than 3. The next lemma explains what happens to Z in this process.

\begin{lemma} \label{decrease}
 	Let $G_i \stackrel{m_2} \longrightarrow  G_{i+1}$ be a move in (S2). Then it holds that:
 	\begin{itemize}
 		\item[(a)] $Z(G_{i+1},l) = Z(G_i,l) - 2$ if it is the first move of (S2),
 		\item[(b)] $Z(G_{i+1},l) = Z(G_i,l) $ otherwise.
 	\end{itemize}
 	In particular we have that
 	\[
		Z(G^{(2)},l) = 
		\begin{cases}
			2  & \text{ if } G^{(2)}=G^{(1)} \\
			0  & \text{ otherwise}.
		\end{cases}	
 	\]
\end{lemma}

\begin{proof}
	We observed that a step $G_i \stackrel{m_2} \longrightarrow  G_{i+1}$ in (S2) adds an edge $f$ which completes a 3-cycle with edges $\lbrace l,f, g \rbrace$, so $f(G_{i+1},l,g)= f(G_{i+1},l,f) = 2$ . Since in $G_i$ the edge $g$ is not in any 3-cycle with $l$, $f(G_i,l,g)=4$, but notice that $f(G_{i+1},l,l')= f(G_i,l,l')$ for any $l' \neq g$ edge of $G_i$, giving us (using \ref{sommacontributi})
	\[
		\sum_{l' \in E(G_i)} f(G_{i+1},l,l')= \sum_{l' \in E(G_i)} f(G_i,l,l') \, \,  -2 = \big| f(G_i,l) \big| -2.
	\]
	Hence, using again \ref{sommacontributi},
	
	\[
		\big|F(G_{i+1},l) \big|= f(G_{i+1},l,f) + \sum_{l' \in E(G_i)} f(G_{i+1},l,l')= \big| F(G_i,l) \big|.
	\]
	
	From this we deduce that the only case when Z changes is when $\mathbb{1}_{E_3(G_{i+1})}(l)=1$ and $\mathbb{1}_{E_3(G_{i})}(l)=0$. This happens exactly when the move $G_i \stackrel{m_2} \longrightarrow  G_{i+1}$ create a first 3-cycle containing $l$, hence when it is the first move of (S2).
\end{proof}

We study for last (S3), the most delicate one. Here a move $G_i \stackrel{m_2} \longrightarrow  G_{i+1}$ adds an edge $f$ not adjacent to $l$ which can either create a 4-cycle containing $l$ or not. In the first case various situations may happens, which lead us to different behaviors of Z. First of all, we notice that if the induced subgraph of $G_{i+1}$ (hence of $G$) with vertex set $V(l) \cup V(f)$ properly contains a 4-cycle, then we have $Z(G_{i+1},l) \geq Z(G_i,l)$, so for convenience we perform these types of moves at the beginning of (S3) and we call this sub process (S3a). After this, we consider another subprocess (S3b), in which we perform all the moves \textbf{m2} which creates induced 4-cycles with respect to its vertices: we call these induced 4-cycles of $G$ containing $l$ the \emph{pages of} $l$, and the set of all the pages of $l$ will be denoted by $\mathcal{B}(l)$. It's important to observe that the creation of a new page $P$ with edges $ \lbrace l,f,g,h \rbrace$ lends us to three different situations respect to the behavior of Z.
\begin{itemize}
	\item[i)]  $f(G_i,l,g)= f(G_i,l,h) = 2$. Then we have that $Z(G_{i+1},l) = Z(G_i,l) + f(G_{i+1},l,f)= Z(G_i,l) + 2$.
	
	Observe that the only possible way in which we end up to this case are when both $g$ and $h$ are in different 4-cycles of $G_i$. We illustrates two examples which will appear in the future.
	
	\begin{figure}[!ht]
		\begin{tikzpicture}[scale=0.6]
			
			\node at (-3.5,2) 	   	 	{$\bullet$};
			\node at (6-3.5,4)			{$\bullet$};
			\node at (6-3.5,0)			{$\bullet$};
			\node at (2.5-3.5,0) 		{$\bullet$};
			\node at (2.5-3.5,4) 		{$\bullet$};
			\node at (-2-3.5,2) 		{$G_i \, =$};
			
			\node at (0+7,2) 	   	 	{$\bullet$};
			\node at (6+7,4)			{$\bullet$};
			\node at (6+7,0)			{$\bullet$};
			\node at (2+7.5,0) 			{$\bullet$};
			\node at (2.5+7,4) 			{$\bullet$};
			\node at (8.5+7,2)			{$ = G_{i+1}$};
			\node at (17,2) 			{};
			\draw (7-3.5,2) edge[->] node[above] {$m_2$} (9.7-3.5,2);
			
			\path
			(2.5-3.5,0) edge node {\hspace{-4mm}$l$} (2.5-3.5,4)
			(2.5-3.5,0) edge node {} (0-3.5,2)
			(0-3.5,2) edge node {} (2.5-3.5,4)
			(2.5-3.5,4) edge node[above] {$g$} (6-3.5,4)
			(2.5-3.5,0) edge node[below] {$h$} (6-3.5,0)
			(0-3.5,2) edge node {} (6-3.5,0)
			(0-3.5,2) edge node {} (6-3.5,4);
			
			\path
			(2.5+7,0) edge node {\hspace{-4mm}$l$} (2.5+7,4)
			(2.5+7,0) edge node {} (0+7,2)
			(0+7,2) edge node {} (2.5+7,4)
			(2.5+7,4) edge node[above] {$g$} (6+7,4)
			(2.5+7,0) edge node[below] {$h$} (6+7,0)
			(0+7,2) edge node {} (6+7,0)
			(0+7,2) edge node {} (6+7,4)
			(6+7,0) edge node[left] {$f$} (6+7,4)
			;
		\end{tikzpicture}
		\caption{$g$ and $h$ are in different 4-cycles of $G_i$ that are not pages} \label{fig:casoia}
	\end{figure}
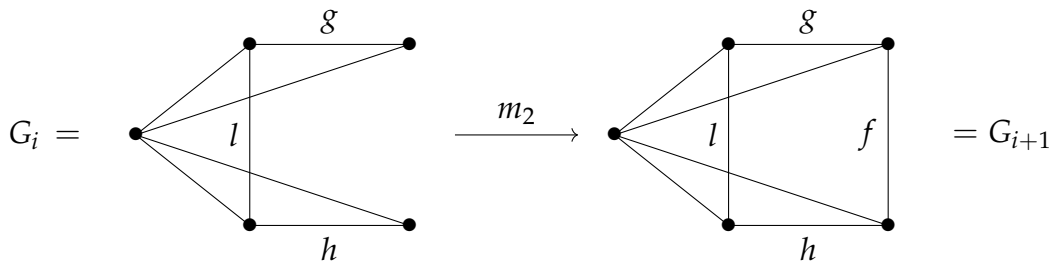
	
	\begin{figure}[!ht]
		\begin{tikzpicture}[scale=0.5]
			
			\node at (0,0) 		{$\bullet$};	
			\node at (0,4) 	    {$\bullet$};
			\node at (4,-1)  	{$\bullet$};
			\node at (4,3) 		{$\bullet$};
			\node at (-4,-1) 	{$\bullet$};
			\node at (-4,3)		{$\bullet$};
			\node at (-6,1.5)		{$G_i =$};
			
			\path
			(0,-0.1) edge node {\hspace{-4mm}$l$} (0,4)
			(0,0.1) edge node {} (4.1,-1)
			(0,4.1) edge node[below] {$g$} (4.1,3.1)
			(4,-1) edge node {} (4,3)
			(0,0.1) edge node[below] {$h$} (-4,-1)
			(-4,-1) edge node {} (-4,3)
			(-4,3.1) edge node {} (0,4.1);
			
			\draw (5,1.5) edge[->] node[above] {$m_2$} (8,1.5);
			
			\node at (0+13,0) 		{$\bullet$};	
			\node at (0+13,4) 	    {$\bullet$};
			\node at (4+13,-1)  	{$\bullet$};
			\node at (4+13,3) 		{$\bullet$};
			\node at (-4+13,-1) 	{$\bullet$};
			\node at (-4+13,3)		{$\bullet$};
			\node at (1+13,1)		{$f$};
			\node at (19.5,1.5)		{$ = G_{i+1}$};
			
			\path
			(0+13,-0.1) edge node {\hspace{-4mm}$l$} (0+13,4)
			(0+13,0.1) edge node {} (4.1+13,-1)
			(0+13,4.1) edge node[below] {$g$} (4.1+13,3.1)
			(4+13,-1) edge node {} (4+13,3)
			(0+13,0.1) edge node[below] {$h$} (-4+13,-1)
			(-4+13,-1) edge node {} (-4+13,3)
			(-4+13,3.1) edge node {} (0+13,4.1)
			(-4+13,-0.9)edge node {} (4+13,3.1);
			
		\end{tikzpicture}
		\caption{$g$ and $h$ are in different pages of $G_i$} \label{fig:casoib}
	\end{figure}
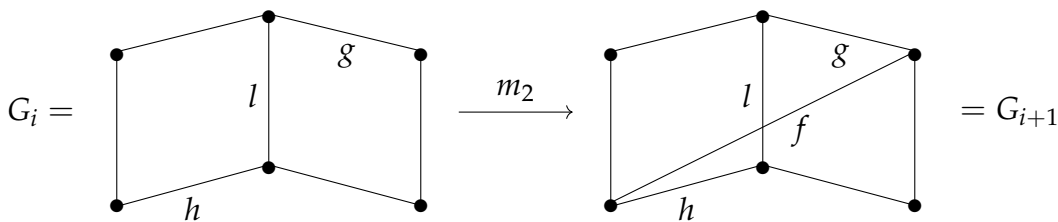
	\item[ii)] $f(G_i,l,g)=2$ and $f(G_i,l,h)= 4$ (or vice versa). Then we have that $f(G_{i+1},l,h)= 2$, so
	\[
		Z(G_{i+1},l) = Z(G_i,l) - 2 + f(G_{i+1},l,f) = Z(G_i,l).
	\] 
	Notice that this is the case when $h$ is not involved in any 3 or 4-cycle of $G_i$, but $g$ is in a 4-cycle of $G_i$.
	
		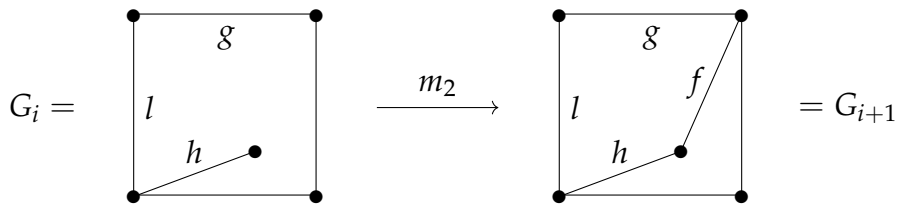
\begin{figure}[!ht]
		\begin{tikzpicture}[scale=0.4]
			
			\node at (0,0) (A) 		{$\bullet$};	
			\node at (0,6) (B)	    {$\bullet$};
			\node at (6,0) (C) 		{$\bullet$};
			\node at (6,6) (D)		{$\bullet$};
			\node at (4,1.5) (E)	{$\bullet$};
			\node at (-3,3) (E)	{$G_i =$};
			
			\path
			(0,-0.1) edge node[right] {$l$} (0,6.1)
			(0,0.1) edge node {} (6.1,0.1)
			(0,6.1) edge node[below] {$g$} (6.1,6.1)
			(6,0.1) edge node {} (6,6.1)
			(0,0.05) edge node[above] {$h$} (4,1.55);
			
			\draw (8,3) edge[->] node[above] {$m_2$} (12,3);
			
			\node at (0+14,0) (A) 		{$\bullet$};	
			\node at (0+14,6) (B)	    {$\bullet$};
			\node at (6+14,0) (C) 		{$\bullet$};
			\node at (6+14,6) (D)		{$\bullet$};
			\node at (4+14,1.5) (E)		{$\bullet$};
			\node at (9.5+14,3) (E)		{$= G_{i+1}$};
			
			\path
			(0+14,-0.1) edge node[right] {$l$} (0+14,6.1)
			(0+14,0.1) edge node {} (6.1+14,0.1)
			(0+14,6.1) edge node[below] {$g$} (6.1+14,6.1)
			(6+14,0.1) edge node {} (6+14,6.1)
			(0+14,0.05) edge node[above] {$h$} (4+14,1.55)
			(4+14,1.55) edge node {\hspace{-4mm}$f$} (6+14,6.1);
		\end{tikzpicture}
		\caption{An example of case ii)} \label{fig:casoii}
	\end{figure}
    
	\item[iii)] $f(G_i,l,g)= f(G_i,l,h) = 4$. Then we have that $f(G_{i+1},l,g) = f(G_{i+1},l,h)= 2$, hence 
	\[
		Z(G_{i+1},l) = Z(G_i,l) - 4 + f(G_{i+1},l,f) = Z(G_i,l) - 2.
	\]
	This is the case when $g$ and $h$ are not involved in any 3 or 4-cycle of $G_i$.
	
	\begin{figure}[!ht]
		\begin{tikzpicture}[scale=0.5]
			
			\node at (0,0) 	   	 	{$\bullet$};
			\node at (0,4)			{$\bullet$};
			\node at (4,0) 			{$\bullet$};
			\node at (4,4) 			{$\bullet$};
			\node at (-4,0) 	   	{$\bullet$};
			\node at (-4,4) 	   	{$\bullet$};
			\node at (-6,2) 		{$G_i \, =$};
			
			\node at (9,0) 	   	 		{$\bullet$};
			\node at (9,4) 	   	 		{$\bullet$};
			\node at (13,0) 			{$\bullet$};
			\node at (13,4) 			{$\bullet$};
			\node at (17,0) 			{$\bullet$};
			\node at (17,4) 			{$\bullet$};
			\node at (19.5,2) 		{$= \, G_{i+1} $};
			
			\draw (5,2) edge[->] node[above] {$m_2$} (8,2);
			
			\path
			(0,0) edge node[left] {$l$} (0,4)
			(4,0) edge node[above] {$h$} (0,0)
			(4,4) edge node[below] {$g$} (0,4)
			(0,0) edge node {} (-4,0)
			(-4,0) edge node {} (-4,4)
			(-4,4) edge node {} (0,4);
			
			\path 
			(13,4) edge node[left] {$l$} (13,0)
			(13,0) edge node {} (9,0)
			(9,0) edge node  {} (9,4)
			(9,4) edge node  {} (13,4)
			(13,0) edge node[above] {$h$} (17,0)
			(17,0) edge node[left] {$f$} (17,4)
			(17,4) edge node[below] {$g$} (13,4);
		\end{tikzpicture}
		\caption{An example of case iii)} \label{fig:casoiii}
	\end{figure}
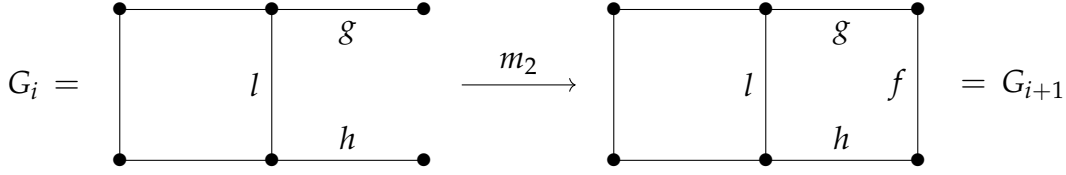
	
\end{itemize}

 Finally, after (S3b) we take the moves which do not create pages of $l$ for last, finishing (S3) with the subprocess (S3c). It's better to observe that in a move of (S3c) we are in the same situation as in (S5), so we have $Z(G_{i+1},l)= Z(G_i,l) + 4$.
Let's give a look at the behavior of Z during the process (S3) step by step. By what we said above on the subprocesses (S3a), given $G^{(3a)}$ the graph obtained at the end of (S3a), it holds that
\begin{equation} \label{disuguaglianzaS3a}
	Z(G^{(3a)},l) \geq 2 - 2 \cdot \mathbb{1}_{E_3}(l),
\end{equation}

We want to find a lower bound for $Z(G^{(3b)},l)$, where $G^{(3b)}$ is the graph obtained at the end of (S3b). In truth, we derive a lower bound which holds for every subgraph involved in this subprocess. In order to do this we define an equivalence relation $\approx$ on the set of all 4-cycles of $G$ containing $l$.

Given $P,Q $ such cycles, we say that $P \sim Q$ if they share exactly two edges (see Figure \ref{fig:esempiosim} for an example). Taking the transitive closure of $\sim$ we obtain the equivalence relation $\approx$. Precisely, $P \approx Q$ if there exists a sequence of 4-cycles $P=P_1, P_2, \ldots, P_n= Q$ in $G$ such that $P_i \sim P_{i+1}$ for all $i \in [n-1]$.

\begin{figure}[!ht]
	\begin{tikzpicture}[scale=0.5]
		
		\node at (0,0) (A) 		{$\bullet$};	
		\node at (0,6) (B)	    {$\bullet$};
		\node at (6,0) (C) 		{$\bullet$};
		\node at (6,6) (D)		{$\bullet$};
		\node at (4,1.5) (E)		{$\bullet$};
		
		\path
		(0,-0.1) edge node[right] {$l$} (0,6.1)
		(0,0.1) edge node[below] {$r$} (6.1,0.1)
		(0,6.1) edge node[below] {$g$} (6.1,6.1)
		(6,0.1) edge node[right] {$s$} (6,6.1)
		(0,0.05) edge node[above] {$h$} (4,1.55)
		(4,1.55) edge node {\hspace{-4mm}$f$} (6,6.1);
	\end{tikzpicture}
	\caption{An example of $P \sim Q$, where $P= \lbrace l, g, h, f \rbrace$ and $Q = \lbrace l,g, r, s \rbrace $} \label{fig:esempiosim}
\end{figure}
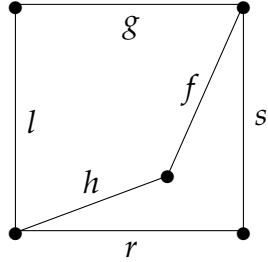

Moreover, we denote by $B(G,l)$ the number of the equivalence classes on the set of all 4-cycles of $G$ containing $l$ which do not contain a 4-cycle created in (S3a). 

\begin{lemma} \label{disuguaglianzaS3c}
	Let $G_i$ be a subgraph that is in the subprocess (S3b). Then it holds that 
	
	\begin{equation} \label{boundS3c}
		Z(G_i,l) \geq 2 - 2 \cdot \mathbb{1}_{E_3}(l) - 2B(G_i,l) .
	\end{equation}
\end{lemma}

\begin{proof}
	We prove it by induction on the number of moves after the subprocess (S3a). The base step is when we are right at the end of (S3a), hence $B(G_i,l)=0$ and so the inequality above is exactly inequality \ref{disuguaglianzaS3a}.
	
	Now let $G_{i+1}$ be in (S3b) obtained by adding an edge $f$ in $G_i$, which is also in (S3b). We observe that $B(G_{i+1},l)$ can have three possible behaviors in function of $B(G_i,l)$.
	\begin{enumerate}
		\item $B(G_{i+1},l) = B(G_i,l) + 1$. This means that the page in $G_{i+1}$ created by $f$ forms a different class, hence the move $G_i \stackrel{m_2} \longrightarrow  G_{i+1}$ is of the type illustrated by Figure \ref{fig:casoiii}, so we have that $Z(G_{i+1},l)=Z(G_i,l)-2$.
		\item  $B(G_{i+1},l) = B(G_i,l)$. This happens exactly when $f$ creates a page $P$ such that $P \sim Q$, for a page $Q$ in $G_i$, so the move $G_i \stackrel{m_2} \longrightarrow  G_{i+1}$ is of the type of Figure \ref{fig:casoii}, hence we have that $Z(G_{i+1},l) \geq Z(G_i,l)$.
		\item $B(G_{i+1},l) = B(G_i,l) - 1$. This is the case when either the page $P$ created by $f$ make a collapse of two different equivalence classes of pages in $G_i$ in one in $G_{i+1}$, or there exists a 4-cycle $Q'$ in $G^{(3a)}$ such that $P \sim Q'$. Both of these scenarios are particular instances that fall in the cases of Figure \ref{fig:casoia} and Figure \ref{fig:casoib}, so we have that $Z(G_{i+1},l)=Z(G_i,l) + 2$.
	\end{enumerate}
	
	Applying the induction hypothesis on $G_i$ it's clear that the inequality \ref{boundS3c} is satisfied.
\end{proof}

Now, putting together what we said about the subprocesses of (S3), we obtain a lower bound on $Z(G^{(3)},l)$.

\begin{lemma} \label{disuguaglianzaG3}
	The following inequality holds
	\begin{equation} \label{boundG3}
		Z(G^{(3)},l) \geq 2 - 2 \cdot \mathbb{1}_{E_3} -2B + 4C,
	\end{equation}
	where $B=B(G,l)$ and $C$ is the number of moves in (S3c).
\end{lemma}

Now it suffices to integrate the relations of (S4) and (S5) to derive the bound on $Z(G,l)$.

\begin{proposition} \label{disuguaglianzaG}
	For any $G=(V,E)$ connected graph and for any $l \in E$ it holds
	\begin{equation} \label{boundG}
		Z(G,l) \geq 2 - 2 \cdot \mathbb{1}_{E_3} -2B + 4N,
	\end{equation}
	where $B$ is as in Lemma \ref{disuguaglianzaG3} and $N$ is the number of moves in (S3c) and in (S5).
\end{proposition}

We are now ready to prove Proposition \ref{inequalityedge}.

\begin{proof}[Proof of Proposition \ref{inequalityedge}:]
	Among all the edges of $G$, let us choose $l= \lbrace v, w \rbrace \in E$ which maximizes $B=B(G,l)$. If $B=0$ the result is immediate by Proposition \ref{disuguaglianzaG}.
	Suppose now that $B>0$; this implies that there exists a page $P$ that is constructed in (S3b) respect to $l$ and it is not equivalent to any 4-cycle constructed in (S3a), and take the edge $l'$ of $P$ not adjacent to $l$. Proposition \ref{disuguaglianzaG} gives the inequality \ref{boundG} for $l$ and the following inequality for $l'$
	\begin{equation} \label{boundGprimo}
		Z(G,l') \geq 2 - 2 \cdot \mathbb{1}_{E_3} -2B' + 4N'.
	\end{equation}
	We notice that any page $Q$ with edges $\lbrace l,g,h,f \rbrace$ (say $f$ is not adjacent to $l$) that is not in the same class of $P$ respect to $\approx$ in $\mathcal{B}(l)$ is such that $f$ is not in a 4-cycle of $G$ containing $l'$, or $f$ is in a 4-cycle of $G$ containing $l'$ but $g$ and $h$ are not. In fact, if $f$ is in a 4-cycle of $G$ that contains $l'$, then $G$ should contain one of the two possible subgraphs illustrated in Figure \ref{fig:sottografi_dim_prop3.4}.
	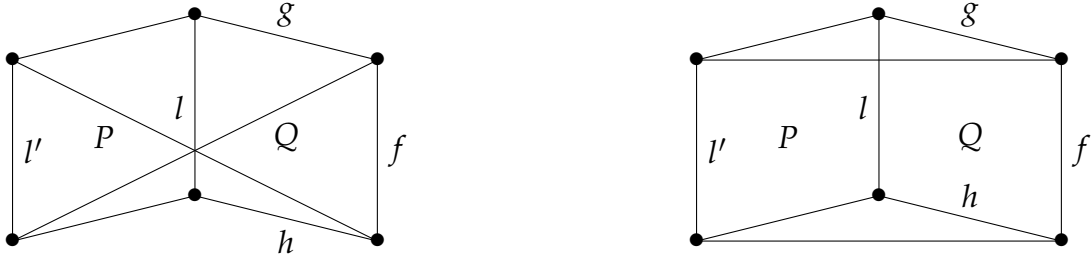
\begin{figure}[!ht]
		\begin{tikzpicture}[scale=0.6]
			
			\node at (0,0) 		{$\bullet$};	
			\node at (0,4) 	    {$\bullet$};
			\node at (4,-1)  	{$\bullet$};
			\node at (4,3) 		{$\bullet$};
			\node at (-4,-1) 	{$\bullet$};
			\node at (-4,3)		{$\bullet$};
			\node at (-2,1.3)	{$P$};
			\node at (2, 1.3) 	{$Q$};
			
			\path
			(0,-0.1) edge node {\hspace{-4mm}$l$} (0,4)
			(0,0) edge node[below] {$h$} (4,-1)
			(0,4) edge node[above] {$g$} (4,3)
			(4,-1) edge node[right] {$f$} (4,3)
			(0,0) edge node {} (-4,-1)
			(-4,-1) edge node[right] {$l'$} (-4,3)
			(-4,3) edge node {} (0,4)
			(-4,-1) edge node {} (4,3)
			(-4,3) edge node {} (4,-1);
			
			\node at (0+15,0) 		{$\bullet$};	
			\node at (0+15,4) 	    {$\bullet$};
			\node at (4+15,-1)  	{$\bullet$};
			\node at (4+15,3) 		{$\bullet$};
			\node at (-4+15,-1) 	{$\bullet$};
			\node at (-4+15,3)		{$\bullet$};
			\node at (-2+15,1.3)	{$P$};
			\node at (2+15, 1.3) 	{$Q$};
			
			\path
			(0+15,-0.1) edge node {\hspace{-4mm}$l$} (0+15,4)
			(0+15,0) edge node[above] {$h$} (4+15,-1)
			(0+15,4) edge node[above] {$g$} (4+15,3)
			(4+15,-1) edge node[right] {$f$} (4+15,3)
			(0+15,0) edge node {} (-4+15,-1)
			(-4+15,-1) edge node[right] {$l'$} (-4+15,3)
			(-4+15,3) edge node {} (0+15,4)
			(-4+15,3) edge node {} (4+15,3)
			(-4+15,-1) edge node {} (4+15,-1);
			
		\end{tikzpicture}
		\caption{The two possible subgraphs contained in $G$} \label{fig:sottografi_dim_prop3.4}
	\end{figure}

The first case would imply $P \approx Q$ so it cannot happen. In the other case, note that $g$ and $h$ cannot be in 4-cycles containing $l'$, because otherwise we get an edge that witnesses $P \approx Q$ or that $P$ is not a page, in contrast with our assumptions.		

As a consequence we have that $N' \geq B-1$ and, vice versa, $N \geq B'-1$. Moreover, if $l \in E_3$, then any 3-cycle of $G$ containing $l$ has an edge $g$ which is not in any 4-cycle containing $l'$: in fact, if this doesn't hold, then locally $G$ contains the subgraph given in Figure \ref{fig:contributoE3}, and this would imply that $P$ is equivalent to a 4-cycle created in (S3a), in contrast with our assumptions.

\begin{figure}[!ht]
	\begin{tikzpicture}[scale=0.7]
		
		\node at (0+7,2) 	   	 	{$\bullet$};
		\node at (6.5+7,4)			{$\bullet$};
		\node at (6.5+7,0)			{$\bullet$};
		\node at (2+7.5,0) 			{$\bullet$};
		\node at (2.5+7,4) 			{$\bullet$};
		\node at (4.5+7,2)			{$P$};
		
		\path
		(2.5+7,0) edge node {\hspace{-4mm}$l$} (2.5+7,4)
		(2.5+7,0) edge node {} (0+7,2)
		(0+7,2) edge node {} (2.5+7,4)
		(2.5+7,4) edge node {} (6.5+7,4)
		(2.5+7,0) edge node {} (6.5+7,0)
		(0+7,2) edge node {} (6.5+7,0)
		(0+7,2) edge node {} (6.5+7,4)
		(6.5+7,0) edge node[left] {$l'$} (6.5+7,4)
		;
	\end{tikzpicture}
	\caption{The local structure of $G$ if $g$ and $h$ are in 3 or 4-cycles of $G$ containing $l'$.} \label{fig:contributoE3}
\end{figure}
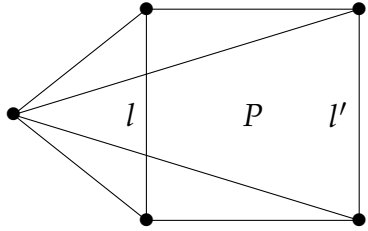
This says that if $l \in E_3$, then $N' \geq B$ (and the analogue if $l' \in E_3$). This leads us to bound $Z(G,l)+ Z(G,l')$ in all the possible cases:
\begin{enumerate}
	\item $l, l' \in E_3$: 
	 \[
	 	Z(G,l)+ Z(G,l') \geq 4(B + B') -2(B+ B')=2(B+B') \geq 2.
	 \]
	\item $l \in E_3$, $l' \not \in E_3$ (and vice versa):
	\[
		Z(G,l) + Z(G,l') \geq 4(B+ B' - 1) -2(B + B' -1)= 2(B + B' -1) \geq 0.
	\]
	\item $l,l' \not \in E_3$: from $l' \not \in E_3$ we have that $B'\geq 1$, so
	\[
		Z(G,l) + Z(G,l') \geq 4(B+B'-2) -2(B+B'-2)= 2(B+B'-2) \geq 0.
	\]
\end{enumerate}
We have to show now that for any edge $f$ different to $l$ it holds that $Z(G,f) \geq 0$. We  notice that the argument above holds for any edge $f$ in any page contained in one of the $B$ special classes, giving in particular $Z(G,f) \geq 0$. For the other edges $f$, we say that for any of the $B$ special equivalence classes of pages of $l$ there is an edge which is not in any 4-cycle containing $f$. In order to show this, let's pick a page $P$ with edges $\lbrace l,g,h,s \rbrace$ ($s$ not adjacent to $l$) in a class mentioned above and let's treat the various cases which depends on the type of the edge $f$.

Let $f$ be in a 3-cycle with $l$. If $g$ or $h$ are in a 3 or 4-cycle with $f$, then $G$ contains one of the two subgraphs illustrated in Figure \ref{fig:sottografi2_dim_prop3.4} (up to isomorphism).

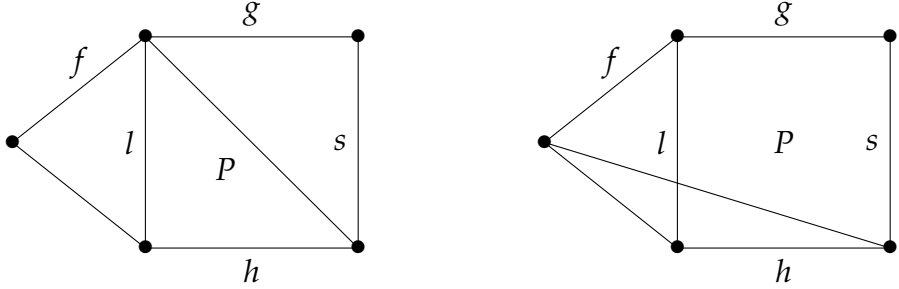
\begin{figure}[!ht]
	\begin{tikzpicture}[scale=0.7]
		
		\node at (0,2) 	   	 	{$\bullet$};
		\node at (6.5,4)			{$\bullet$};
		\node at (6.5,0)			{$\bullet$};
		\node at (2.5,0) 			{$\bullet$};
		\node at (2.5,4) 			{$\bullet$};
		\node at (4,1.5)			{$P$};
		
		\path
		(2.5,0) edge node {\hspace{-4mm}$l$} (2.5,4)
		(2.5,0) edge node {} (0,2)
		(0,2) edge node[above] {$f$} (2.5,4)
		(2.5,4) edge node[above] {$g$} (6.5,4)
		(2.5,4) edge node {} (6.5,0)
		(2.5,0) edge node[below] {$h$} (6.5,0)
		(6.5,0) edge node[left] {$s$} (6.5,4);
		
		\node at (0+10,2) 	   	 	{$\bullet$};
		\node at (6.5+10,4)			{$\bullet$};
		\node at (6.5+10,0)			{$\bullet$};
		\node at (2.5+10,0) 			{$\bullet$};
		\node at (2.5+10,4) 			{$\bullet$};
		\node at (4.5+10,2)			{$P$};
		
		\path
		(2.5+10,0) edge node {\hspace{-4mm}$l$} (2.5+10,4)
		(2.5+10,0) edge node {} (0+10,2)
		(0+10,2) edge node[above] {$f$} (2.5+10,4)
		(2.5+10,4) edge node[above] {$g$} (6.5+10,4)
		(2.5+10,0) edge node[below] {$h$} (6.5+10,0)
		(0+10,2) edge node {} (6.5+10,0)
		(6.5+10,0) edge node[left] {$s$} (6.5+10,4)
		;
	\end{tikzpicture}
	\caption{The two possible subgraphs contained in $G$} \label{fig:sottografi2_dim_prop3.4}
\end{figure}

The first says that $P$ is not a page and the second gives $P \sim Q$, where $Q$ is the 4-cycle containing $f$ and $h$, so it's constructed in (S3a), against the assumptions given on $P$.

With the same argument it can be proven that for any 4-cycle $Q$ that contains $f$ and $l$, if $g$ and $h$ are in 3 or 4-cycles that contain $f$ then $P \approx Q$, and this settle down all the other cases when $f$ is in a 4-cycle of $G$.

If $f \in E(G^{(3)})$ is added in (S3c), then there exist two edges $r$ and $t$ incident to $f$ which insist to the same vertex of $l$, say $v$. But now, if there is a 4 cycle of $G$ containing $f$ and $g$, then there exists an edge $e$ incident to $f$ which contains the vertex $w$, and so there would be a 4-cycle with edges $f,e,l$ and one between $r$ and $t$, in contrast with our assumption on $f$.

When $f$ is not an edge of $G^{(3)}$ we have to distinguish two different situations, whether $f$ is incident to $s$ or not. When $f$ is incident to $s$ (hence without loss of generality $f$ is also incident to $h$) if there is a 4-cycle in $G$ containing both $f$ and $g$, then $G$ contains one of the two subgraphs given by Figure \ref{fig:sottografi3_dim_prop3.4}.

\begin{figure}[!ht]
	\begin{tikzpicture}[scale=0.7]
		
		\node at (6.5,4)			{$\bullet$};
		\node at (6.5,0)			{$\bullet$};
		\node at (2.5,0) 			{$\bullet$};
		\node at (2.5,4) 			{$\bullet$};
		\node at (4,1.5)			{$P$};
		\node at (9,2)				{$\bullet$}; 
		
		\path
		(2.5,0) edge node {\hspace{-4mm}$l$} (2.5,4)
		(2.5,4) edge node[above] {$g$} (6.5,4)
		(2.5,4) edge node {} (6.5,0)
		(2.5,0) edge node[below] {$h$} (6.5,0)
		(6.5,0) edge node[left] {$s$} (6.5,4)
		(6.5,0) edge node[below] {$f$} (9,2)
		(9,2) edge node {} (6.5,4);
		
		\node at (9+10,2) 	   	 	{$\bullet$};
		\node at (6.5+10,4)			{$\bullet$};
		\node at (6.5+10,0)			{$\bullet$};
		\node at (2.5+10,0) 		{$\bullet$};
		\node at (2.5+10,4) 		{$\bullet$};
		\node at (4.5+10,2)			{$P$};
		
		\path
		(2.5+10,0) edge node {\hspace{-4mm}$l$} (2.5+10,4)
		(6.5+10,0) edge node[below] {$f$} (9+10,2)
		(9+10,2) edge node {} (2.5+10,4)
		(2.5+10,4) edge node[above] {$g$} (6.5+10,4)
		(2.5+10,0) edge node[below] {$h$} (6.5+10,0)
		(6.5+10,0) edge node[left] {$s$} (6.5+10,4)
		;
	\end{tikzpicture}
	\caption{The two possible situations if both $f$ and $g$ are in a common 4-cycle of $G$} \label{fig:sottografi3_dim_prop3.4}
\end{figure}
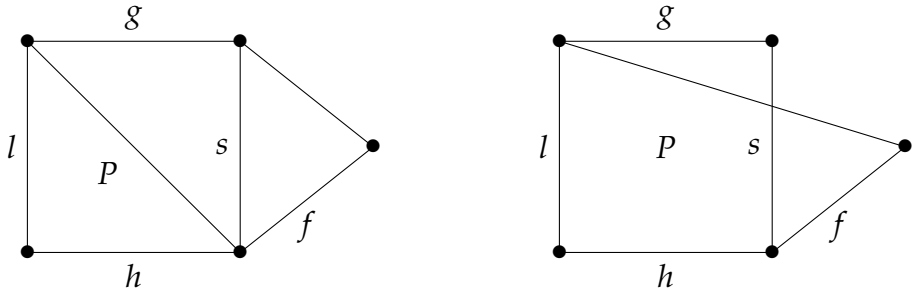

The first says that $P$ is not a page, while the second would imply that $f \in E(G^{(3)})$ against our assumptions.

When $f$ is not incident to $s$, if we suppose that $g,h$ and $s$ are in some 4-cycles of $G$ containing $f$, then we have that there is a 4-cycle containing $f$ and $l$ (which contradicts that $f \not \in E(G^{(3)})$) or $G$ contains the subgraph given by Figure \ref{fig:sottografi4_dim_prop3.4}.

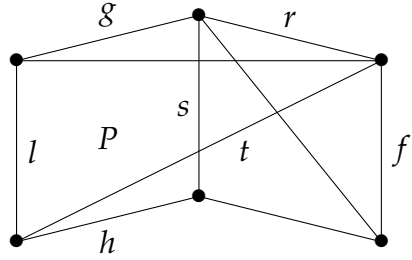
\begin{figure}[!ht]
	\begin{tikzpicture}[scale=0.6]
		
		\node at (0,0) 		{$\bullet$};	
		\node at (0,4) 	    {$\bullet$};
		\node at (4,-1)  	{$\bullet$};
		\node at (4,3) 		{$\bullet$};
		\node at (-4,-1) 	{$\bullet$};
		\node at (-4,3)		{$\bullet$};
		\node at (-2,1.3)	{$P$};
		\node at (1,1)		{$t$};
		
		\path
		(0,-0.1) edge node {\hspace{-4mm}$s$} (0,4)
		(0,0) edge node {} (4,-1)
		(4,-1) edge node {} (0,4)
		(0,4) edge node[above] {$r$} (4,3)
		(4,-1) edge node[right] {$f$} (4,3)
		(0,0) edge node[below] {$h$} (-4,-1)
		(-4,-1) edge node[right] {$l$} (-4,3)
		(-4,3) edge node[above] {$g$} (0,4)
		(-4,3) edge node {} (4,3)
		(-4,-1) edge node {} (4,3);
		
	\end{tikzpicture}
	\caption{The subgraph contained in $G$ in a case when $f$ is not incident to $s$} \label{fig:sottografi4_dim_prop3.4}
\end{figure}
This would imply that $P \approx Q$ where $Q$ is the induced subgraph with vertices given by the union of the vertices of $l,g$ and $r$, again in contrast with the assumptions on $P$.

Hence in all the cases cited above we can construct a sequence of subgraphs starting from $ \lbrace f \rbrace$ ending in $G$ such that there are at least $B$ edges added in the processes (S3c) or (S5), giving us
\[
	Z(G,f) \geq -2B(G,f) + 4N(f) \geq -2B + 4B \geq 2.
\]

\end{proof}

\section{Characterization of graphs with $z_2(G)=0$ }\label{sec:z_2(G)=0}

The proof of Proposition \ref{inequalityedge} also helps us to characterize all the graphs $G$ such that $z_2(G)=0$. Let $G$ be such a graph. Given $l \in E$ which maximizes $B= B(G,l)$, by the lower bounds on $Z(G,l) + Z(G,l')$ we derive that $B \leq 1$, and if we consider also the lower bound given at the end of the proof, for any edge $f$ of $G$ with $B(G,f)=1$ we have that $f \not \in E_3$. Now thanks to Proposition \ref{disuguaglianzaG} we deduce that $Z(G,f)=0$ for any $f \in E$, and in order to obtain this it is necessary that $G=G^{(3b)}(f)$ for any $f \in E$, so for any pair of $f,f'$ of different edges in $G$ we have that there exists a 3 or 4-cycle which contains $f$ and $f'$. The following two lemmas treats separately the cases $B=1$ and $B=0$.
\begin{lemma}
	Let $G$ be a connected graph with $n$ vertices such that $z_2(G)=0$ and there exists $l \in E$ with $B(G,l)=1$. Then $G \cong K_{2,n-2}$.
\end{lemma}
\begin{proof}
	By what we said above, since there are no 3-cycles that contain $l$, for any edge $f$ of $G$ there is a page of $l$ containing $f$, and by $B(G,l)=1$ we have that all the pages of $l$ are equivalent. Notice that for any $G$ and $l$ such that $B(G,l)=1$ there exists a sequence of subgraphs starting from $\lbrace l \rbrace$ ending in $G$ in which there is a $j>0$ such that $Z(G_i,l)= 2$ for any $i <j$, $Z(G_j,l)=0$, and $Z(G_i,l) \leq Z(G_{i+1},l)$ for any $i \geq j$. We can construct such a sequence in this way: once done the process (S1), we start (S3) with a move $G_{j-1} \stackrel{m_2} \longrightarrow  G_j$ which constructs a page $P$, and after this we perform all the moves \textbf{m2} which create pages $Q$ such that $Q \sim P$. Then inductively, we perform a move $G_i \stackrel{m_2} \longrightarrow  G_{i+1}$ if the page $Q'$ created in $G_{i+1}$ is such that $Q' \sim Q$ for a certain $Q$ page in $G_i$. In particular, by the hypothesis $z_2(G)=0$, we have also that $Z(G_i,l)=0$ for any $i \geq j$.
	
	We show that for any pages $Q,Q'$ of $l$ we have $Q \sim Q'$. In fact, if there are $Q, Q'$ such that $Q \not \sim Q'$, by $Q \approx Q'$ we have that there exist $Q_1, \ldots, Q_k$ ($k \geq 1)$ pages such that $Q=Q_0 \sim Q_1 \sim \cdots \sim Q_k \sim Q_{k+1}=Q'$. Without loss of generality we suppose that $Q \sim P \sim Q'$ for a certain page $P$, hence we consider a sequence of subgraphs as above in which the first three pages constructed are, in order, $P$, $Q$ and $Q'$. Thus $G$ has a subgraph of the form given by Figure \ref{QsimPsimQ'}.
	
	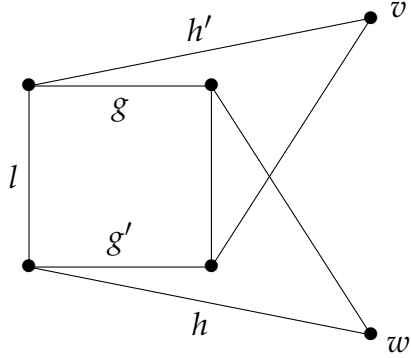
\begin{figure}[!ht]
		\begin{tikzpicture}[scale=0.6]
			
			\node at (10,-1.5) 	   	 	{$\bullet$};
			\node at (10,5.5)			{$\bullet$};
			\node at (10.6,5.7)			{$v$};
			\node at (10.6,-1.7)			{$w$};
			\node at (6.5,4)			{$\bullet$};
			\node at (6.5,0)			{$\bullet$};
			\node at (2.5,0) 		{$\bullet$};
			\node at (2.5,4) 		{$\bullet$};
			
			\path
			(2.5,0) edge node {\hspace{-4mm}$l$} (2.5,4)
			(2.5,0) edge node[below] {$h$} (10,-1.5)
			(2.5,4) edge node[above] {$h'$} (10, 5.5)
			(6.5,0) edge node {} (10,5.5)
			(10,-1.5) edge node {} (6.5,4)
			(2.5,4) edge node[below] {$g$} (6.5,4)
			(2.5,0) edge node[above] {$g'$} (6.5,0)
			(6.5,0) edge node[left] {} (6.5,4)
			;
		\end{tikzpicture}
		\caption{$Q \sim P \sim Q'$ with $Q \not \sim Q'$, where $\lbrace l, g ,g' \rbrace \subseteq P$, $\lbrace l, g ,h \rbrace \subseteq Q$, $\lbrace l, g' ,h' \rbrace \subseteq Q'$.} \label{QsimPsimQ'}
	\end{figure} 
	By the hypothesis $h$ and $h'$ have to stay in a 4-cycle of $G$, hence $ \lbrace v, w \rbrace \in E$. But given the move $G_i \stackrel{m_2} \longrightarrow  G_{i+1}$ that adds $\lbrace v, w \rbrace$, it can be easily seen that $Z(G_{i+1},l) \geq Z(G_i,l) + 2$, hence $Z(G,l)>0$, in contrast with our hypothesis.
	
	By the fact that $Q \sim Q'$ for any $Q,Q'$ pages of $l$ it follows that there is an edge $s$ incident to $l$ such that any page $P$ of $l$ contains also $s$. Since any edge of $G$ is in a page of $l$ we have that, given $v, w $ the two non-common vertices of $l$ and $s$, $\lbrace x,v \rbrace$, $\lbrace x,w \rbrace \in E$ for any $x \in V \smallsetminus \lbrace v, w \rbrace$, hence $K_{2,n-2} \subseteq G$. Moreover notice that any edge of the form $ f=\lbrace x, y \rbrace$ with $x,y \in V \smallsetminus \lbrace v, w \rbrace$ can't be in $G$ because $l,s,f$ can't be in a 4-cycle, hence we conclude.
\end{proof}

\begin{lemma}
	Let $G$ be a connected graph with $n$ vertices ($ n \geq 3$) such that $z_2(G)=0$ and $B(G,l)=0$ for any $l \in E$. Then $G \cong K_{1,1,n-2}$ or $G \cong K_n$.
\end{lemma}
\begin{proof}
	Since $B(G,l)=0$ for any $l\in E$, then $E_3=E$, so $n \geq 3$.
	If $G \not \cong K_{1,1,n-2}$, then for any $l \in E$ there is $f \in E$ which is in a 4-cycle containing $l$ but not in any 3-cycle of $G$ containing $l$, and for convenience $f$ is not incident to $l$. We want to show that for any $l= \lbrace v, w \rbrace \in E$ and for any $x \in V \smallsetminus \lbrace v,w \rbrace$, $G$ has the 3-cycle $ \lbrace v,w,x \rbrace$. In fact, if it doesn't hold, then there exists an $l= \lbrace v, w \rbrace$ and a vertex $u$ such that $ \lbrace u,v \rbrace \in E$ but $\lbrace u, w \rbrace \not \in E$. Since $G=G^{(3b)}(l)$ there exists a 4-cycle containing $l,g$ with vertices $\lbrace u,v,w,x \rbrace$, and since $B(G,l)=0$ we can suppose that such 4-cycle is not a page, so $f= \lbrace v,x \rbrace \in E(G)$. Using that $G \not \cong K_{1,1,n-2}$, we find an $h = \lbrace y,z \rbrace \in E$ not adjacent to $f$, and again by $B(G,f)=0$ we can suppose that $f$ and $h$ are contained in a 4-cycle that is not a page. Depending on the intersection of $\lbrace u,w \rbrace$ and $\lbrace y,z \rbrace$ we fall into two cases.
	
	\begin{itemize}
		\item[i)] $\lbrace u,w \rbrace \cap \lbrace y,z \rbrace \neq \emptyset$. Then $G$ contains a subgraph $H$ with $5$ vertices given in Figure \ref{interseznonvuota}.
	
	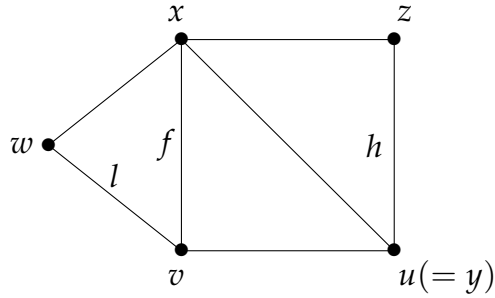
\begin{figure}[!ht]
		\begin{tikzpicture}[scale=0.7]
			
			\node at (0,2) 	   	 	{$\bullet$};
			\node at (-0.5,2)		{$w$};
			\node at (6.5,4)			{$\bullet$};
			\node at (6.7,4.5)		{$z$};
			\node at (6.5,0)			{$\bullet$};
			\node at (7.5,-0.5)		{$u (=y)$};
			\node at (2.5,0) 			{$\bullet$};
			\node at (2.4,-0.5)		{$v$};
			\node at (2.5,4) 			{$\bullet$};
			\node at (2.4,4.5)		{$x$};
			
			\path
			(2.5,0) edge node {\hspace{-4mm}$f$} (2.5,4)
			(2.5,0) edge node[above] {$l$} (0,2)
			(0,2) edge node {} (2.5,4)
			(2.5,4) edge node {} (6.5,4)
			(2.5,4) edge node {} (6.5,0)
			(2.5,0) edge node {} (6.5,0)
			(6.5,0) edge node[left] {$h$} (6.5,4);
			
		\end{tikzpicture}
		\caption{$\lbrace u,w \rbrace \cap \lbrace y,z \rbrace \neq \emptyset$.} \label{interseznonvuota}
	\end{figure}
	
	\item[ii)] $\lbrace u,w \rbrace \cap \lbrace y,z \rbrace = \emptyset$. Then $G$ has a subgraph with $6$ vertices given in Figure \ref{intersezvuota}.
	
		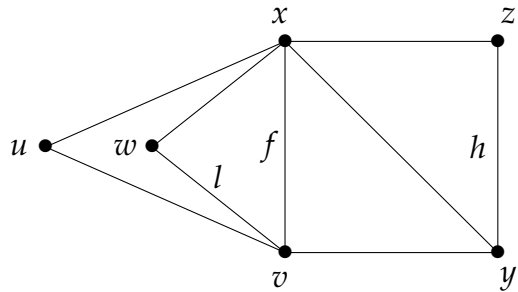
\begin{figure}[!ht]
		\begin{tikzpicture}[scale=0.7]
			
			\node at (-2,2) 		{$\bullet$};
			\node at (-2.5,2)		{$u$};
			\node at (0,2) 	   	 	{$\bullet$};
			\node at (-0.5,2)		{$w$};
			\node at (6.5,4)			{$\bullet$};
			\node at (6.7,4.5)		{$z$};
			\node at (6.5,0)			{$\bullet$};
			\node at (6.7,-0.5)		{$y$};
			\node at (2.5,0) 			{$\bullet$};
			\node at (2.4,-0.5)		{$v$};
			\node at (2.5,4) 			{$\bullet$};
			\node at (2.4,4.5)		{$x$};
			
			\path
			(-2,2) edge node {} (2.5,4)
			(-2,2) edge node {} (2.5,0)
			(2.5,0) edge node {\hspace{-4mm}$f$} (2.5,4)
			(2.5,0) edge node[above] {$l$} (0,2)
			(0,2) edge node {} (2.5,4)
			(2.5,4) edge node {} (6.5,4)
			(2.5,4) edge node {} (6.5,0)
			(2.5,0) edge node {} (6.5,0)
			(6.5,0) edge node[left] {$h$} (6.5,4);
			
		\end{tikzpicture}
		\caption{$\lbrace u,w \rbrace \cap \lbrace y,z \rbrace = \emptyset$.} \label{intersezvuota}
	\end{figure}
	\end{itemize}
	We observe that, given $G'$ an induced subgraph of $G$ with 5 vertices which contains an isomorphic copy of $H$ as subgraph, for any $e$ edge of $G'$ we can construct a sequence of subgraphs starting from $\lbrace e \rbrace$ ending in $G$ such that $Z(G_j,e)=0$ for some $j >0$, $Z(G_i,e)$ is weakly increasing for any $i \geq j$ and $Z(G_k,e) = Z(G',e)$ for some $k>j$; the construction is given by performing (S1), (S2), and then we add all the remaining edges of $G'$ at the beginning of (S3).
	
	We can verify with pure computations that a graph $G'$ with 5 vertices that contains (an isomorphic copy of) $H$ is such that $Z(G',e)=0$ for any $e \in E(G')$ if and only if $G'$ is the complete graph $K_5$. So, if we come back to the case i), the induced subgraph of $G$ with vertices $\lbrace u,v,w,x,z \rbrace$ has to be $K_5$, in contrast with the fact that $\lbrace u, w \rbrace \not \in E$. We end up to the same contradiction in the case ii), applying this argument first with the induced subgraph with vertices $\lbrace v,w,x,y,z \rbrace$, then with the subgraph with vertex set $ \lbrace u,v,w,x,y \rbrace$.
	
	Once we have that for any $l= \lbrace v, w \rbrace \in E$ and for any $x \in V \smallsetminus \lbrace v,w \rbrace$ the 3-cycle $ \lbrace v,w,x \rbrace$ is in $G$, it is simple to show that $G \cong K_n$.
\end{proof}
\begin{corollary}\label{cor: Z = 0}
    Let $G=([n],E)$ be a connected graph. Then $z_2(G)=0$ if and only if either $G\cong K_n$, $G\cong K_{1,1,n-2}$ or $G\cong K_{2,n-2}$.
\end{corollary}

\section{Connection to the Ehrhart theory of symmetric edge polytopes} \label{sec:Ehrhart}

We now turn our attention to the Ehrhart theory of symmetric edge polytopes. Inspired by \cite{AJKKV23} our goal is to investigate the effect of deleting an edge in the graph $G$ on the $h^*$-polynomials $h^*_{P_G}(t)$. We thus fix an edge $e= \lbrace i,j \rbrace \in E$ and study the relationship between the polytopes $P_G$ and $P_{G \sm e}$. The idea is to find the "right" choice of edge $e$ to delete, which will allow us to control the changes of the $h^*$-polynomials.

Since the dimension of a symmetric edge polytope depends only on number of vertices and connected components of the graph, if there is a cycle in $G$ that contains $e$, then $P_{G \sm e}$ has the same dimension of $P_G$. Thus $h^*_{P_G}(t) - h^*_{P_{G \sm e}}(t)$ remains palyndromic and we can formulate the following conjecture.

\begin{conjecture}
\label{conj:reduction}
Let $G$ be a 2-connected graph. There exists an edge $e \in E$ such that $h^*_{P_G}(t) - h^*_{P_{G \sm e}}(t)$ is $\gamma$-nonnegative.
\end{conjecture}

This is enough to prove $\gamma$-positivity of the $h^*$-polynomial of symmetric edge polytopes as it has been exploited in \cite{AJKKV23} for the quadratic coefficient, as we prove in the following proposition.
\begin{proposition}
\label{prop:goal_implies_gamma}
\Cref{conj:reduction} implies the Oshugi--Tsuchiya conjecture (\Cref{conj: OT}).
\end{proposition}
\begin{proof}
Since every graph decomposes uniquely into its 2-connected components and  \cite[Proposition~5.2]{OT21} implies that the corresponding $h^*$-polynomial behaves multiplicative with respect to this decomposition, we can reduce \Cref{conj: OT} to the case of 2-connected graphs.

We apply induction on the cyclomatic number 
\[
\cy(G) := \abs{E} - \abs{V} + 1
\]
of a 2-connected graph $G = (V,E)$. For the base case, we assume that $\cy(G) = 0$. Thus, $G$ is a tree and the corresponding $h^*$-polynomial is known to be $\gamma$-nonnegative (see, for example, \cite[Example~5.1]{OT21}). 

Now, pick an edge $e \in E$ that satisfies \Cref{conj:reduction}. Since
\[
\cy(G) = \cy(G\sm{e}) + 1
\]
(see, for example, the proof of \cite[Theorem~3.2]{AJKKV23}), the induction hypothesis implies that the $h^*$-polynomial of $P_{G\sm{e}}$ is $\gamma$-nonnegative. The proof follows from \Cref{conj:reduction} and the fact that $\gamma$-positivity is preserved by taking linear combinations.
\end{proof}

In the rest of this section, we will give a formula for $h^*_{P_G}(t) - h^*_{P_{G\sm ij}}(t)$ in terms of a nice triangulation of $P_G \sm P_{G\sm ij}$, and use this formula to prove a first step in the direction of \Cref{conj:reduction}: that is, we can always find an edge so that the second gamma coefficient is non-negative.

\subsection{The $\ell^*$-polynomial of special simplices}\label{par:local_h_special_simplices}

Before we can understand the behaviour of $P_G \sm P_{G\sm ij}$, we turn our attention to a special family of simplices contained in there. We are in particular interested in their $\ell^*$-polynomials, since they will play a fundamental role in our use of the Betke-McMullen formula (\Cref{BetkeMcMullen}) in the next subsection.

Let $e=\{i,j\}$ be an edge of $G$ and $\sigma$ a unimodular simplex contained in a facet of $P_{G\setminus e}$ visible from the point $e_{ij}:= e_j -e_i$. By Corollary \ref{cor:unimodsimplexchar}, the associated oriented subgraph $G(\sigma)$ is an oriented spanning tree inside $G \sm e$. We now define the (non necessarily unimodular) simplex $\hat{\sigma}:=\text{conv}\{\sigma \cup \{e_{ij}\}\}$, obtained by adding the vertex $e_{ij}$ to $\sigma$. The oriented subgraph $G(\hat{\sigma})$ contains a unique cycle, namely $\vec{P}_\sigma\cup\vec{ij}$, where $\vec{P}_\sigma$ is contained in $G(\sigma)$.
    
We are interested in computing the $h^*$- and $\ell^*$-polynomials of $\hat{\sigma}$. To do so, we introduce the following parameters. 

\begin{align*}
    a_\sigma :&= \text{ number of edges of }\vec{P}_\sigma \text{ oriented against }\vec{ij} \text{ when traversing the cycle } \vec{P}_\sigma \cup \vec{ij},\\
    b_\sigma :&= \text{ number of edges of }\vec{P}_\sigma \text{ oriented like }\vec{ij},\\
    l_\sigma :&= a_\sigma + b_\sigma = \text{ length of }\vec{P}_\sigma, \\
    h_\sigma :&= a_\sigma - b_\sigma = \text{ number of lattice hyperplanes parallel to } \sigma \text{ between } e_{ij} \text{ and }\sigma.
\end{align*}

\begin{figure}
    \centering
\begin{tikzpicture}
	\node[label=above:{$i$}] (i) at (0,2) 	  	{$\bullet$};
	\node[label=above:{$j$}] (j) at (1,2) 	   	{$\bullet$};
    \node[label=left:{$i_1$}] (1) at (-.5,.5) 	   	{$\bullet$};    
    \node[label=below:{$i_2$}] (2) at (.4,.1) 	   	{$\bullet$};   
    \node[label=right:{$i_3$}] (3) at (.9,.9) 	   	{$\bullet$};
    \draw (i) edge[->] (j)
    (i) edge[->, red] (1)
    (1) edge[->, red] (2)
    (3) edge[->, blue] (2)
    (3) edge[->, red] (j);
\end{tikzpicture}
    \caption{Unique cycle in $T_\sigma \cup \vec{ij}$. Edges in red contribute to $a_\sigma$, while blue to $b_\sigma$.}
    \label{fig:placeholder}
\end{figure}
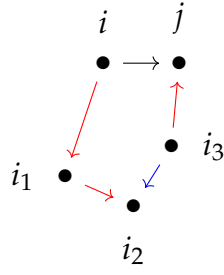

\begin{remark}
    Note that the condition that $e_{ij}$ is visible from $\sigma$ guarantees that $h_\sigma \geq 1$, precisely because it can be interpreted as the number of lattice hyperplanes between $e_{ij}$ and $\sigma$.  This follows from \Cref{thm:facetchar}: if the equation for the facet on which $\hat{\sigma}$ is lying is $\sum_{v\in V} f(v)x_v =1$, then the point $e_{ij}$ lies on the translate of this hyperplane with right hand side equal to $f(j)-f(i)$. Since the orientation of an edge in $\sigma$ corresponds to whether the function $f$ is increasing or decreasing by $1$ from the source vertex to the target vertex, we have that $f(j)-f(i)=a_\sigma - b_\sigma$.
\end{remark}

With these in hand, we can give a formula for the $\ell^*$-polynomial of $\hat{\sigma}$.

\begin{lemma}\label{lem:box_pol}
    Let $\sigma$ be a unimodular simplex contained in a facet of $P_{G\setminus e}$ visible from the point $e_{ij}$. Let $\hat{\sigma}=\text{conv}\{\sigma \cup \{e_{ij}\}\}$. Then
    \[
    	\ell^*_{\hat{\sigma}}(t)=\begin{cases}
    		 t^{b_\sigma + 2} + \dots + t^{a_\sigma -1} & \text{if } G(\hat{\sigma}) \text{ is a cycle and }a_\sigma - b_\sigma \geq 3\\
    		 0 & \text{otherwise}.
    	\end{cases}
    \]   
\end{lemma}
\begin{proof}
    Let us denote by $\vec{E}(\sigma)$ the collection of oriented edges of the tree $G(\sigma)$. Then we can express the fundamental parallelepiped as follows 
    \begin{equation}\label{eq:fund_parallel}
    \overset{\circ}{\Pi}(\hat{\sigma}) = \left\{ \left(\lambda_{ij}(e_j-e_i) + \sum_{\vec{kl} \in \vec{E}(\sigma)} \lambda_{kl} (e_l - e_k), \, \lambda_{ij} + \sum_{\vec{kl} \in \vec{E}(\sigma)} \lambda_{kl} \right) \in \RR^{d+1} \Bigg | \,\lambda_{ij}, \lambda_{kl} \in (0,1)\right\}.
    \end{equation}

    If $\vec{kl}$ is an edge ending in a leaf of $G(\hat{\sigma})$, say $l$ is the leaf, then the $l$-th coordinate of any point in $\overset{\circ}{\Pi}(\hat{\sigma})$ is $\lambda_{kl}$, since $e_l$ only appears in the summand $e_l-e_k$. Since $\lambda_{kl} \in (0,1)$, it follows that there are no integer points in $\overset{\circ}{\Pi}(\hat{\sigma})$ and thus $\ell^*(\hat{\sigma})=0$. 

    Thus the only simplices $\hat{\sigma}$ which have $\ell^*(\hat{\sigma})\neq 0$ are those for which the corresponding subgraph $G(\hat{\sigma})$ is a cycle. In this case, let $i_0=i, i_1, \dots, i_d=j, i_0$ be the underlying cycle.  Without loss of generality, suppose that the sequence $i_0, i_1, \dots, i_d$ is just $0, 1, \dots, d$. Consider a point $p \in \overset{\circ}{\Pi}$. Then $e_t$ appears twice in the sum from \ref{eq:fund_parallel} defining $p$, once with the coefficient $\pm \lambda_{t-1,t}$ and once with $\pm \lambda_{t, t+1}$. If the edges $\{t-1, t\}$ and $\{t, t+1\}$ appear with opposite orientations in $G(\hat{\sigma})$, then the $t$-th coordinate of $p$ is $p_t = \lambda_{t-1,t}+\lambda_{t,t+1}$, while if the two edges appear with the same orientation, we have $p_t=\lambda_{t-1,t}-\lambda_{t,t+1}$. Thus we have $p \in \ZZ^{d+1}$ if and only if 
    \begin{itemize}
        \item the coefficient of all edges oriented like $\vec{ij}$ in the cycle $G(\hat{\sigma})$ is $\lambda \in (0,1)$, and
        \item the coefficient of edges oriented against $\vec{ij}$ is $1-\lambda$, and
        \item $\lambda \in \{\frac{1}{a_\sigma -b_\sigma -1}, \frac{2}{a_\sigma -b_\sigma -1}, \dots, \frac{a_\sigma -b_\sigma -2}{a_\sigma -b_\sigma -1}\}$. 
    \end{itemize}

    The last condition is obtained by imposing that the last coordinate of $p$ is an integer, which holds if and only if $\lambda (a_\sigma -b_\sigma -1) \in \ZZ$. In order for $\lambda$ to lie in the interval $(0,1)$, we further need that $a_\sigma - b_\sigma \geq 3$.
\end{proof}

\subsection{A formula for the difference of $h^*$-polynomials when removing an edge} \label{par:computationhstar}

In order to compute $h^*_{P_G}(t)-h^*_{P_{G\setminus ij}}(t)$ using the Betke-McMullen formula (\Cref{BetkeMcMullen}), we must triangulate the set
$P_G \sm P_{G \sm ij}$. $P_G$ contains exactly two vertices which $P_{G\sm ij}$ does not, namely $e_{ij}$ and $e_{ji}$. Thus $P_G \sm P_{G \sm ij}$ consists of two connected components, symmetric with respect to the origin, one being the cone of $e_{ij}$ with the facets of $P_{G\sm ij}$ visible from $e_{ij}$.

Consider then a unimodular triangulation $\Gamma^{ij}$ of the complex of facets
of $P_{G\sm ij}$ visible from $e_{ij}$induced by an HJM triangulation of $P_{G\sm ij}$. We can then define $\Delta^{ij}:= \Gamma^{ij} * e_{ij}$, that is, the simplicial complex whose maximal
simplices are the convex hull of a maximal simplex of $\Gamma^{ij}$ together with $e_{ij}$. This simplicial complex indeed triangulates the closure of the connected component of $P_G \sm P_{G \sm ij}$ containing $e_{ij}$ and can therefore be used to express the difference of $h^*$-polynomials.
Note that by construction all simplices in $\Delta^{ij}$ are simplices of the form described in
\Cref{par:local_h_special_simplices}. This allows us to express the difference of the $h^*$-polynomials
after deleting an edge as follows.

\begin{lemma}\label{lem:diff_hstar}
    Let $G$ be a graph and $\Gamma^{ij}$ and $\Delta^{ij}$ be the triangulations described above. Then we have
    \[
    h^*_{P_G}(t)-h^*_{P_{G\setminus ij}}(t) = 2th_{\Gamma^{ij}}(t) + 2\sum_{\substack{F \in \Delta^{ij} \\ G(F) \text{ cycle with } \\ a_F -b_F \geq 3}} h_{\lk_{\Delta^{ij}}(F)}(t)(t^{b_F+2} + \dots + t^{a_F-1}).
    \]
\end{lemma}
\begin{proof}
We denote the support of the complex $\Delta^{ij}$ by $|\Delta^{ij}|$.
The set $P_G \setminus P_{G\setminus ij}$ is composed of two symmetric pieces, one of which can be written as $|\Delta^{ij}| \setminus |\Gamma^{ij}|$. Thanks to the additivity of the Ehrhart series, and because all these objects are $n$-dimensional except for $\Gamma^{ij}$, which is $n-1$ dimensional, we have the following equality of their $h^*$-polynomials:

\[
    \frac{1}{2}(h^*_{P_G}(t)-h^*_{P_{G\setminus ij}}(t)) = h^*_{|\Delta^{ij}|}(t)-(1-t)h^*_{|\Gamma^{ij}|}(t) = h^*_{|\Delta^{ij}|}(t)-(1-t)h_{\Delta^{ij}}(t), 
\]
where the second equality follows since $\Gamma^{ij}$ is a unimodular triangulation (so $h^*_{|\Gamma^{ij}|}(t) = h_{\Gamma^{ij}}(t)$), and because $\Delta^{ij}$ is a cone over $\Gamma^{ij}$, we have that $h_{\Delta^{ij}}(t) = h_{\Gamma^{ij}}(t)$.

Applying Betke-McMullen (\Cref{BetkeMcMullen}) and \Cref{lem:box_pol} we have 
\begin{equation*}
    \begin{aligned}
        h^*_{|\Delta^{ij}|}(t) &= \sum_{F \in \Delta^{ij}} h_{\lk_{\Delta^{ij}}(F)}(t)\ell^*_F(t) \\
        &= h_{\Delta^{ij}}(t) + \sum_{\substack{F \in \Delta^{ij} \\ G(F) \text{ cycle with } \\ a_F -b_F \geq 3}} h_{\lk_{\Delta^{ij}}(F)}(t)\ell^*_F(t)\\
        &= \sum_{\substack{F \in \Delta^{ij} \\ G(F) \text{ cycle with } \\ a_F -b_F \geq 3}} h_{\lk_{\Delta^{ij}}(F)}(t)(t^{b_F+2} + \dots + t^{a_F-1})
    \end{aligned}
\end{equation*}
where the $h_{\Delta^{ij}}$ in the second line comes from $F=\emptyset \in \Delta^{ij}$, since by convention $\ell^*_\emptyset(t)=1$.  
By combining the two equations, we obtain the desired formula.
\end{proof}

\subsection{Connection to $z_2(G)$}

We now highlight a connection of \Cref{lem:diff_hstar} with the lower bound of the number of edges of $P_G$ discussed in \Cref{sec:z_2(G)}. 

Let $G=([n], E)$ be a $2$-connected graph, and consider the polynomials $c^{ij}(t):=h^*_{P_G}(t)-h^*_{P_{G\setminus ij}}(t)$, where $ij \in E$. Note that since $G$ is $2$-connected, $c^{ij}(t)$ is still a palindromic polynomial with the same center as the two polynomials $h^*_{P_G}(t)$ and $h^*_{P_{G\setminus ij}}(t)$.

An approach to \Cref{conj: OT} is to to prove that there exists an edge $ij$ in $G$ for which the $\gamma$-polynomial associated with $c^{ij}(t)$ has nonnegative coefficients. To this end, we propose to investigate the following stronger statement: consider the sum over the edges $ij$ of $G$ of the polynomials $c^{ij}(t)$. The result is a palindromic polynomial with the same center, whose associated $\gamma$-polynomial is the sum of the $\gamma$-polynomials of each $c^{ij}(t)$. In particular, the $\gamma$-nonnegativity of this sum would certify the existence of an edge $ij\in E$ for which $h^*_{P_G}(t)-h^*_{P_{G\setminus ij}}(t)$ is $\gamma$-nonnegative.
\begin{conjecture} \label{conj: zG nonnegative}
	Let $G=(V,E)$ be a $2$-connected graph. The polynomial
	\[
	Z_G(t):=\sum_{ij\in E}\gamma(c^{ij}(t))=\sum_{ij\in E}\gamma(h^*_{P_G}(t)-h^*_{P_{G\setminus ij}}(t))=\sum_{ij\in E}\gamma_{P_G}(t)-\gamma_{P_{G\setminus ij}}(t)
	\]
	has nonnegative coefficients.
\end{conjecture}
In particular, this conjecture implies \Cref{conj: OT}. We have checked the validity of \Cref{conj: zG nonnegative} for all $2$-connected graphs with up to $8$ vertices using a computer.\\

\begin{example}
	Let $G=K_n$, for some $n\geq 3$. In this case $G\setminus ij$ is, up to isomorphism, independent on the choice of $ij \in E$. We can determine the facets of $P_{G\setminus ij}$ which are visible from $e_{ij}$. By \Cref{thm:facetchar} facets of $P_{G\setminus ij}$ correspond to certain maps $f:[n]\to\mathbb{Z}$. The ones visible from $e_{ij}$ are precisely the maps in \Cref{thm:facetchar} for which $f(j)-f(i)>1$. For $G=K_n$, by condition $(i)$ in \Cref{thm:facetchar} this is equivalent to $f(j)-f(i)=2$, which in turn implies that $f(j)-f(k)=f(k)-f(i)=1$ for every $k\in [n]\setminus{i,j}$. In particular there is, up to translation, only one map satisfying the conditions in \Cref{thm:facetchar}, and hence there is a unique facet $F$ of $P_{G\setminus ij}$ visible from $e_{ij}$, $F=\text{conv}\{ e_i-e_k,e_k-e_j~:~ k\in[n]\setminus\{i,j\}\}$. 
	In order to compute $h^*_{P_G}(t)-h^*_{P_{G\setminus ij}}(t)$,
	we use a HJM triangulation of $F$ described in \Cref{thm:HJMtriang}. Fix an order on the edges of $G\setminus ij$ and assume that $iv$ is the minimum among all the edges of the form $ik$ and $kj$ with $k\in [n]\setminus\{i,j\}\}$. The case in which the minimal edge is of the form $vj$ is completely analogous. The maximal simplices in the corresponding unimodular triangulation $\Gamma^{ij}$ of $F$ are the convex hull of subsets $S\subseteq \{ e_i-e_k,e_k-e_j~:~ k\in[n]\setminus\{i,j\}\}$ such that:
	\begin{itemize}
		\item[-] $e_i-e_v\in S$ and $e_v-e_j\in S$;
		\item[-] exactly one between $e_i-e_k$ and $e_k-e_j$ is in $S$, for every $k\in  [n]\setminus\{i,j,v\}$.
	\end{itemize}
	In particular, $\Gamma^{ij}$ is combinatorially equivalent to the join of a $(n-4)$-dimensional cross-polytope and a $1$-simplex. Hence $h_{\Gamma^{ij}}(t)=(1+t)^{n-3}$. By \Cref{lem:diff_hstar} we have that 
	\[
	h^*_{P_G}(t)-h^*_{P_{G\setminus ij}}(t) = 2th_{\Gamma^{ij}}(t)=2t(1+t)^{n-3}.
	\]
	Since $\gamma(2t(1+t)^{n-3})=2t$, we conclude that $Z_{K_n}(t)=\binom{n}{2}2t=n(n-1)t$.
\end{example}
\begin{example}
	Let $G=C_n$, be the $n$-cycle for some $n\geq 3$. Then its $h^*$ is computed in \cite{OT21}.
	Deleting any edge yields a tree on $n$ vertices, whose symmetric edge polytope is unimodularly equivalent to the standard cross polytope $\Diamond_{n-1}$, and therefore its $h^*$-polynomial is $h^*_{\Diamond_{n-1}}=(t+1)^{n-1}$. 
	
	We thus obtain 
	\begin{align*}
		Z_{C_n}(t)=& \,
	n(h^*_{P_{C_n}}(t)-h^*_{\Diamond_{n-1}}(t))\\
	=& \, n\sum_{i=1}^{\lfloor\frac{n-1}{2}\rfloor} \binom{2i}{i}t^i(t+1)^{n-1-2i},
	\end{align*}
	which is $\gamma$-nonnegative.
\end{example}

We now want to show how one can use \Cref{lem:diff_hstar} to give explicit expressions for the first coefficients of $c^{ij}=\sum_{0}^d a_i t^i$, and then combine these to give expressions for the first coefficients of $Z_G$. We easily see that $a_0=0$ and $a_1=2$. With a bit more work we can find $a_2$. To do so, we need intermediate results about the triangulation $\Gamma^{ij}$.

\begin{proposition} \label{lem:minimalpath}
	Let $G=([n],E)$ be a 2-connected graph. Let $\gamma$ be one of the shortest paths in $G$ from $i$ to $j$, and let $\Gamma$ be the HJM triangulation induced by an order $<$ for which the edges of $\gamma$ are smaller than any edge not in $\gamma$ Then, any induced oriented path $\vec{\gamma}$ is in  $\Gamma$.
\end{proposition}

\begin{proof}
	We have to prove that $\prod_{e \in \vec{\gamma}} p_e \not \in \text{in}_< (I_{P_G})$.
	Let $k$ be the length of $\gamma$. There cannot be a $2k$-cycle or a $(2k-1)$-cycle that contains $\gamma$ which produces a binomial of type (1) or (2) in Proposition \ref{thm:HJMtriang} with leading term $\prod_{e \in \vec{\gamma}} p_e$, because of the minimality of $\gamma$.
	
	Suppose now that $\prod_{e \in \vec{\gamma}} p_e$ is the multiple of a generator of $\text{in}_< (I_{P_G})$. By Proposition \ref{thm:HJMtriang}, there exists a subset $S$ of edges of $\vec{\gamma}$ such that $\prod_{e \in S} p_e \in \text{in}_< (I_{P_G}) $. 
	
	Let $v_1$ and $v_2$ be respectively the first and the last vertices in $\gamma$ which are also in $C_S$, then consider the two paths $\tau_1$, $\tau_2$ from $v_1$ to $v_2$ that form $C_S$. By the minimality of $\gamma$, $\ell(\tau_i) \geq \ell(\gamma_{|v_1,v_2})$ for $i=1,2$, and since $S \subseteq \gamma_{|v_1,v_2}$ we have that $|C_S|= \ell(\tau_1) + \ell(\tau_2)  \geq 2|S|$, hence we are in the case (1) of Proposition \ref{thm:HJMtriang}, and this implies that
	$\ell(\tau_1)= \ell(\tau_2)= \ell(\gamma_{|v_1,v_2})$ and $S=\gamma_{|v_1,v_2}$.

	Let $l$ be the minimal edge that appears in $C_S$ and suppose without loss of generality that $l \in \tau_1$; by the hypothesis $l \not \in S$. 
	Now the path that goes from $i$ to $v_1$ by $\gamma$, from $v_1$ to $v_2$ by $\tau_1$ and from $v_2$ to $j$ by $\gamma$ would be less than $\gamma$, in contrast with the assumptions on the minimality of $\gamma$.
\end{proof}

We now focus on paths oriented entirely from $i$ to $j$.

\begin{lemma}\label{b_F=0}
    Let $\Gamma^{ij}$ be the unimodular triangulation of the facets of $P_{G\sm ij}$ visible from $e_{ij}$. Then there is a unique simplex $F \in \Gamma^{ij}$ which corresponds to a path between $i$ and $j$ completely oriented against $\vec{ij}$, that is, $b_F=0$.
\end{lemma}
\begin{proof}
	Existence follows from \Cref{lem:minimalpath}, which guarantees that the shortest path between $i$ and $j$ oriented from $i$ to $j$ corresponds to a face in the HJM triangulation of $P_{G\sm ij}$; this face is in a visible facet because the path is of length at least $2$ and oriented correctly. We now prove uniqueness of such a face. Suppose to the contrary that there were two different paths oriented from $i$ to $j$ which correspond to faces of $\Gamma^{ij}$. Let $i_1$ and $i_2$ respectively be the first vertex where the paths diverge and where they first meet again, and call $\vec{P_1}$ and $\vec{P_2}$ the restriction of those oriented paths between $i_1$ and $i_2$. Suppose that $\vec{P_1}$ is at most as long as $\vec{P_2}$, and consider the oriented cycle $C$ obtained as the union of $\vec{P_2}$ and the edges of $\vec{P_1}$ oriented inversely. By \Cref{thm:HJMtriang}, if we choose as $I$ any subset of the oriented edges of $\vec{P_2}$ of size $\lfloor|P_1 \cup P_2|/2\rfloor$, the set $I$ forms a non-face of $\Gamma^{ij}$, and therefore so does the original oriented path it is a subset of, which yields a contradiction. 
\end{proof}

We can now give an explicit formula for the quadratic coefficient of $c^{ij}(t)$
\begin{lemma}
Let $G$ be a $2$-connected graph with $n$ vertices and $ij$ an edge in $G$. Then the coefficient of $t^2$ in $c^{ij}(t)$ is
    \begin{equation}
        a_2 =
        \begin{dcases}
          2(f_0(\Gamma^{ij})-n+1),& \text{ if shortest path in } G\setminus ij \text{ has length }2, \\
          2(f_0(\Gamma^{ij})-n+2), & \text{ if shortest path in } G\setminus ij \text{ has length }\geq 3.
        \end{dcases}\\
    \end{equation}
\end{lemma}
\begin{proof}
    We use the formula from \Cref{lem:box_pol} and analyze the quadratic term to obtain the following expression: 
    \[a_2= 2h_1(\Gamma^{ij}) +2\sum_{\substack{F \in \Delta^{ij} \\ G(F) \text{ cycle with } \\ b_F=0,\, a_F \geq 3}} 1 \]

    Thanks to \Cref{b_F=0}, we know that there is a unique $F \in \Delta^{ij}$ with $b_F=0$. We then see that the sum in the expression above is empty if the unique facet $F$ of $\Gamma^{ij}$ with $b_F=0$ satisfies $a_F <3$ and $1$ otherwise. We conclude by observing that $h_1(\Gamma^{ij})=f_0(\Gamma^{ij})-n+1$.
\end{proof}

Because $c^{ij}(t)$ is a palindromic polynomial, we can rewrite it in the $\gamma$-basis; let $\gamma_0, \gamma_1, \dots$ be the coefficients in this basis. Then knowing the three coefficients  $a_0, a_1$ and $a_2$ of the polynomial in the standard monomial basis allows us to derive a formula for $\gamma_0, \gamma_1$ and more interestingly for $\gamma_2$. . 

\begin{equation}
    \begin{aligned}
        \gamma_2 = & a_2 - (a_1 - (n-1)a_0)(n-3) \\
        = & 
            \begin{dcases}
              2(f_0(\Gamma^{ij})-2n+4) ,& \text{ if shortest path in } G\setminus ij \text{ has length }2, \\
              2(f_0(\Gamma^{ij})-2n+5), & \text{ if shortest path in } G\setminus ij \text{ has length }\geq 3.
            \end{dcases}\\
    \end{aligned}
\end{equation}

The next lemma gives the key to interpreting these numbers.

\begin{lemma}
    Let $G$, $e_{ij}$ and $\Gamma^{ij}$ be defined as above. 
    Then $f_0(\Gamma^{ij})$ is equal to the number of neighbors of $e_{ij}$ in $P_G$, the set of which we denote by $N_{P_G}(e_{ij})$.
\end{lemma}

\begin{proof}
Recall that the only lattice points in $P_G$ are its vertices and the origin. Thus the vertices of $\Gamma^{ij}$ are simply the vertices of facets of $P_{G \sm ij}$ visible from $e_{ij}$, which are simply the vertices of $P_{G \sm ij}$ visible from $e_{ij}$. To see that these are just the neighbors of $e_{ij}$ in $P_G$, first observe that neighbors of $e_{ij}$ in $P_G$ are certainly visible visible in $P_{G\sm ij}$ from $e_{ij}$, since being visible means that the segment between $v$ and $e_{ij}$ is not contained in the interior of $P_{G\sm ij}$, which is certainly true for an edge of $P_G$. For the converse inclusion, observe that if $v$ is a vertex of $P_{G \sm ij}$ visible from $e_{ij}$, and thus the segment $S$ between $v$ and $e_{ij}$ is outside $P_{G \sm ij}$, if $S$ were not an edge of $P_G$, then we must have that $v$ is in the convex hull of $e_{ij}$ and certain other vertices of $P_G$, but this is not possible, since it is a vertex of $P_G$.
\end{proof}

Let $Z_G(t)=\sum_{i}z_it^i$. It is easy to verify that $z_0=0$ and $z_1=2|E|$. We now show that the quadratic coefficient $z_2$ of $Z_G(t)$ is tightly related with the bound given in \Cref{sec:z_2(G)}. Indeed we have that

\begin{equation}\label{eq: z2 is z2}
\begin{aligned}
    z_2=\sum_{ij \in E} \gamma_2(c^{ij}(t)) = & \sum_{ij \in E} |N_{P_G}(e_{ij})| - 2|E|(2n-5) -2|E_3| \\
    = & 2f_1(P_G) - 2|E|(2n-5) -2|E_3| = z_2(G),
\end{aligned}
\end{equation}
where $E_3 \subseteq E$ is the set of all edges of $G$ which are contained in a cycle, and $z_2(G)$ the quantity defined in \eqref{def: z_2}. The second equality is thanks to a simple double counting, since each vertex of each edge of $P_G$ is counted once, as the neighbor of the other vertex.
 
\Cref{thm:edges_SEP_ineq} then directly implies the following result, which was originally proved in \cite[Theorem 3.2]{AJKKV23}.

\begin{theorem}
    Let $G=(V,E)$ be a $2$-connected graph. Then there exists an edge $ij\in E$ such that $\gamma_2(h^*_{P_G}) - \gamma_2(h^*_{P_{G\setminus ij}}) \geq 0$.
\end{theorem}

\section{Further directions}
\label{sec:conjectured_formula}

The formula in \Cref{lem:diff_hstar} allows us to write $h^*_{P_G}(t) - h^*_{P_{G\setminus ij}}(t)$ as the sum of polynomials, all of which have nonnegative coefficients. Unfortunately, these summands are not palindromic, so this decomposition does not yield an analog formula for the $\gamma$-polynomial of $h^*_{P_G}(t) - h^*_{P_{G\setminus ij}}(t)$. During our investigation we observed a second intriguing formula for this difference, which we offer as a conjecture. 
	
Let $\Gamma$ be a unimodular triangulation of the subcomplex of $P_{G\setminus ij}$ consisting of all facets visible by $e_{ij}$, and consider

\[
L:= \lbrace h_F \, | \, F \text{ facet of } \Gamma \rbrace=
\lbrace h_1> h_2> \cdots> h_k \rbrace
\]
\\
be the set of all possible lattice distances from $e_{ij}$ to a facet of $\Gamma$ in decreasing order.

For $s=1,\ldots,k$ we define 

\[
\Gamma_s := \langle \sigma \text{ facet of } \Gamma \, | \, h_{\sigma} \geq h_s \rangle,
\]
the subcomplex of $\Gamma$ generated by the facets with lattice distance at least $h_s$ from $e_{ij}$. Observe that by definition $\Gamma_{s}$ is a subcomplex of $\Gamma_{s+1}$, for every $s=1,\dots,k-1$.

The following conjecture predicts that we can write $h^*_{P_G}(t) - h^*_{P_{G \sm ij}}(t)$ as the sum of palindromic polynomials in terms of these subcomplexes.

\begin{conjecture}\label{cong: formula for difference}
	Let $G$ be a $2$-connected graph, let $ij$ be any edge of $G$. Then
	\begin{equation}\label{eq:conj_sum}
			h^*_{P_G}(t) - h^*_{P_{G \sm ij}}(t)  =
			2t \sum_{s=1}^{k}\Bigl(\Bigl( h_{\Gamma_s}(t) - h_{\Gamma_{s -1}}(t)\Bigr)  \sum_{r=0}^{h_s -1} t^r \Bigr),
	\end{equation}
	where we set $h_{\Gamma_0}(t):=0$.\\
	Moreover, for every $s=1,\dots,k$ the polynomial $\Bigl( h_{\Gamma_s}(t) - h_{\Gamma_{s -1}}(t)\Bigr)  \sum_{r=0}^{h_s -1} t^r$ is palindromic, with nonnegative coefficients and with the same center as the left hand side.
\end{conjecture}
In particular, the $\gamma$-polynomial of $h^*_{P_G}(t) - h^*_{P_{G \sm ij}}(t)$ would be the sum of the $\gamma$-polynomials of the summands on the right hand side. We want to underline that the last sentence in \Cref{cong: formula for difference} predicts a rather unexpected behavior, as we are not aware of a geometric reason why $h_{\Gamma_s}(t) - h_{\Gamma_{s -1}}(t)$ should be a palindromic polynomial.\\
Observe that \Cref{eq:conj_sum} is true when $L= \lbrace 1 \rbrace$. In fact, in this case there is no facet of $\Gamma_1=\Gamma^{ij}$ at height higher than one from $e_{ij}$, and hence there is no face $F$ of $\Delta^{ij}$ with $a_F-b_F\geq 2$. It follows by \Cref{lem:diff_hstar} that 
\[
h^*_{P_G}(t) - h^*_{P_{G \sm ij}}(t)  =
2th_{\Gamma^{ij}}(t),
\]
which coincides with \Cref{cong: formula for difference}. We illustrate the conjecture in two examples. We will see in the second that there are cases in which a summand in the right hand side is not $\gamma$-nonnegative although the whole sum is. However we believe that the formula in \eqref{eq:conj_sum} might give a useful way to control how the $\gamma$-polynomial changes after removing an edge in $G$.  

\begin{example}\label{ex:negative-gamma}
	Let $G$ be the graph with vertex set $[5]$ and edges $E=\{12,23,34,45,15,13\}$. It has been observed in \cite{AJKKV23} that the polynomial $c^{13}(t)=h^*_{P_G}(t) - h^*_{P_{G \sm 13}}(t)$ is not $\gamma$-nonnegative. Indeed, the $3$-dimensional complex $\Gamma^{13}$ has 3 facets at lattice distance $1$ from $e_{13}$ and we have that
	\[
		c^{13}(t)=h^*_{P_G}(t) - h^*_{P_{G \sm 13}}(t) = 2t(1+t+t^2)=2t(1+t)^2-2t^2,
	\]
	whose associated $\gamma$-polynomial has a negative coefficients.
	If we consider the edge $12$ then the corresponding complex $\Gamma^{12}$ has $3$ facets all at lattice distance $1$ from $e_{12}$. We can verify that
	\[
	c^{12}(t) = 2t(1+4t+t^2)=2t(1+t)^2+4t^2.
	\]
	The same holds for the edge $23$. Finally, if we consider the edge $34$ we have that $\Gamma^{34}$ has $4$ facets, two at lattice distance $1$ from $e_{34}$ and two at lattice distance $2$ from $e_{34}$. Thus we have an example where the sum in \Cref{eq:conj_sum} has more than one term. As predicted by \Cref{cong: formula for difference} we have that
	\[
	c^{34}(t) = 2t((1+t)(1+t)+(2t))=2t(1+t)^2+4t^2,
	\]
	and note that not just the whole sum, but all the summands of \Cref{eq:conj_sum} are palindromic. The same holds for the edges $45$ and $15$.
\end{example}

We illustrate \Cref{cong: formula for difference} in a larger example.
\begin{example}\label{ex : G on 8v}
	Let $G=([8],E)$ with $E=\{12,16,17,23,34,38,45,56,67,78\}$, and consider the edge $12$. There are 42 facets of $P_{G\setminus 12}$ which are visible from $e_{12}$: 4 have lattice distance $3$, $10$ have lattice distance $2$ and 28 facets have lattice distance $1$ from $e_{12}$. We remark that not all of these facets are unimodular simplices. In \Cref{tab: table} we include the polynomials appearing on the right hand side of \eqref{eq:conj_sum}.\\
	 
	\begin{table}[h]
		\centering
		\begin{tabular}{l|l|c|l|l}
			$s$ & $h_s$ & $\#$ facets of $P_{G\setminus 12}$  & $h_{\Gamma_s}(t)$ & $h_{\Gamma_s}(t)-h_{\Gamma_{s-1}}(t)$\\
			& & at distance $h_s$ from $e_{ij}$ & &\\ \hline
			$1$ & $3$ & $4$ ($1$ is not a simplex) & $1 + 2t + 2t^2 + t^3$ &  $1 + 2t + 2t^2 + t^3$\\
			$2$ & $2$ & $10$ ($4$ are not a simplex) & $1 + 7t + 10t^2 + 6t^3$ &  $5t + 8t^2 + 5t^3$\\
			$3$ & $1$ & $28$ ($10$ are not a simplex) & $1 + 11t + 30t^2+26t^3+4t^4$ &  $4t + 20t^2 + 20t^3+4t^4$
		\end{tabular}
		\vspace{10pt}
		\caption{The polynomials in \Cref{cong: formula for difference} for \Cref{ex : G on 8v}.}
		\label{tab: table}
	\end{table}
	\Cref{cong: formula for difference} predicts correctly that
	\begin{align*}
		\frac{1}{2t}(h^*_{P_G}(t) - h^*_{P_{G \sm 12}}(t)) &= (1 + 2t + 2t^2 + t^3)(1+t+t^2)\\
		&+(5t + 8t^2 + 5t^3)(1+t)\\
		&+(4t + 20t^2 + 20t^3+4t^4)\\
		&=1+12t+38t^2+38t^3+12t^4+t^5.
	\end{align*}
	Observe that while $h^*_{P_G}(t) - h^*_{P_{G \sm 12}}(t)$ is $\gamma$-nonnegative, the first summand is not $\gamma$-nonnegative. 
\end{example}
	The following question is also motivated by computations for graphs with few vertices, in which case the answer is positive. Recall that a pure simplicial complex is \emph{shellable} if there is a total ordering $F_1<\dots <F_m$ of its facets such that the intersection of each $F_i$ with the complex generated by the facets $F_1,\dots,F_{i-1}$ is pure of codimension $1$. Remarkably, the coefficients of the $h$-vector of a shellable simplicial complex have a combinatorial interpretation (see for instance \cite[III.2]{Stanley-greenBook}). This fact could be used to give a combinatorial interpretation for the polynomials $h_{\Gamma_s}(t) - h_{\Gamma_{s -1}}(t)$.
\begin{question}
	Are the simplicial complexes $\Gamma_s$ shellable for any $s= 1, \ldots, k$? Is there a shelling order of $\Gamma$ for which every facet of $\Gamma_s$ is smaller than any facet of $\Gamma_{s+1}$, for every $s=1,\dots,k-1$?
\end{question}

\section*{Acknowledgments}
We thank Michele D'Adderio and Vassilis Dionyssis Moustakas for fruitful discussions. LV is partially supported by the PRIN
project “Algebraic and geometric aspects of Lie theory” (2022S8SSW2) and by INdAM – GNSAGA. GC is partially supported by a DFG grant "Extremal convex bodies with respect to lattice functionals" within the SPP "Combinatorial Synergies".

\bibliographystyle{alpha}
\bibliography{bibliography}

\end{document}